\documentclass[reqno,12pt]{amsart}
\usepackage{amsmath,amssymb,amsthm,graphicx,mathrsfs,url}
\usepackage{wrapfig}
\usepackage{mathtools}
\usepackage[utf8]{inputenc}
 \usepackage[T1]{fontenc}
\usepackage[usenames,dvipsnames]{color}
\usepackage[colorlinks=true,linkcolor=Red,citecolor=Green]{hyperref}
\usepackage[super]{nth}
\usepackage[open, openlevel=2, depth=3, atend]{bookmark}
\hypersetup{pdfstartview=XYZ}
\usepackage[font=footnotesize]{caption}
\usepackage{a4wide}

\usepackage{epstopdf}
 
\usepackage{hyperref}

\DeclarePairedDelimiter\floor{\lfloor}{\rfloor}

\theoremstyle{plain}
\newtheorem{theo}{Theorem}[section]
\newtheorem{lemm}{Lemma}[section]
\newtheorem{prop}{Proposition}[section]

\theoremstyle{definition}

\theoremstyle{remark}
\newtheorem{rema}{Remark}[section]
 
\numberwithin{equation}{section}

\newcommand{\R}{\mathbb{R}}

\newcommand{\D}{\mathbb{D}}
\newcommand{\N}{\mathbb{N}}

\newcommand{\HH}{\mathbb{H}}
\newcommand{\eps}{\varepsilon}
\newcommand{\ke}{\text{ker }}
\newcommand{\wt}{\widetilde}

\newcommand{\be}{\begin{equation}}
\newcommand{\ee}{\end{equation}}

% ...

% ...

\title[Boundary rigidity of negatively-curved asymptotically hyperbolic surfaces]{Boundary rigidity of negatively-curved asymptotically hyperbolic surfaces}

\author{Thibault Lefeuvre}

\address{Laboratoire de Mathématiques d’Orsay, Univ. Paris-Sud, CNRS, Université Paris-Saclay, 91405 Orsay, France}

\email{thibault.lefeuvre@u-psud.fr}

\begin{document}

\begin{abstract}
In the spirit of Otal \cite{jo} and Croke \cite{cc}, we prove that a negatively-curved asymptotically hyperbolic surface is \textit{boundary distance rigid}, where the distance between two points on the boundary at infinity is defined by a renormalized quantity.
\end{abstract}

\maketitle

\section{Introduction}

\subsection{Main result}

We consider $\overline{M}$ a smooth compact connected $n+1$-dimensional manifold with boundary. We say that $\rho : \overline{M} \rightarrow \R_+$ is a \textit{boundary defining function} on $\overline{M}$ if it is smooth and satisfies $\rho = 0$ on $\partial M$, $d\rho \neq 0$ on $\partial M$ and $\rho > 0$ on $M$. Let us fix such a function $\rho$. A metric $g$ on $M$ is said to be \textit{asymptotically hyperbolic} if
\begin{enumerate}
\item the metric $\overline{g} = \rho^2 g$ extends to a smooth metric on $\overline{M}$,
\item $|d\rho|_{\rho^2 g} = 1$ on $\partial M$.
\end{enumerate}
Note that these two properties are independent of the choice of $\rho$ because any other boundary function $\rho_0$ can be written $\rho_0 = e^f \rho$ and $\overline{g_0} = \rho_0^2 g = e^{2f} \rho^2 g$ also extends smoothly on $\partial M$ and satisfies on the boundary:
\[ \left| d(e^f \rho) \right|_{e^{2f} \rho^2 g} = e^{-f} \left| e^f d\rho \right|_{\rho^2 g} = |d\rho|_{\rho^2 g} = 1 \]
However, the extension of the metric $\rho^2 g$ on the boundary, that is $\rho^2 g|_{\partial M}$, is not independent of the choice of $\rho$ but its conformal class is. This conformal class of metrics on $\partial M$ is called the \textit{conformal infinity}. In the rest of the paper, $\overline{M}$ will be two-dimensional, so $\partial M$ will be one-dimensional and in this case, all the metrics are conformally equivalent. As a consequence, this statement is rather pointless but it takes another interest if the manifold has dimension superior or equal to three.

It can be proved that such a manifold admits a canonical product structure in the vicinity of the boundary $\partial M$ (see \cite{crg} for instance), that is, given a metric $h_0$ on $\partial M$, in the conformal infinity of $g$, there exists a smooth set of coordinates $(\rho, y)$ on $\overline{M}$ (where $\rho$ is a boundary defining function) such that $|d \rho|_{\rho^2 g} = 1$ in a neighborhood of $\partial M$ (and not only on $\partial M$), $\rho^2 g|_{T \partial M} = h$. Moreover, on a collar neighborhood near $\partial M$, the metric has the form
\be \label{eq:g} g = \dfrac{d \rho^2 + h(\rho)}{\rho^2}, ~~~ \text{on } (0, \eps) \times \partial M, \ee
for some $\eps > 0$ and where $h(\rho)$ is a smooth family of metrics on $\partial M$ such that $h(0) = h_0$. From this expression, one can prove that the sectional curvatures of $(M,g)$ all converge towards $-1$ as $\rho$ goes to $0$.

$M$ is not a compact manifold and the length of a geodesic $\alpha(x,x')$ joining two points $x$ and $x'$ on the boundary at infinity is clearly not finite. However, in \cite{ggsu}, a \textit{renormalized length} $L(\alpha(x,x'))$ for a geodesic $\alpha(x,x')$ is introduced, which roughly consists in the constant term in the asymptotic development of the length of $\alpha_\eps(x,x') := \alpha(x,x') \cap \left\{\rho \geq \eps\right\}$ as $\eps$ goes to $0$. This yields a new object characterized by the asymptotically hyperbolic manifold $(M,g)$ and one can actually wonder, as usual in inverse problem theory, up to what extent it conversely determines $(M,g)$. Notice that the renormalized length is not independent of the choice of the boundary defining function $\rho$, and thus, neither of the choice of the conformal representative $h_0$ in the conformal infinity.

From now on, we assume that $M$ is a surface with negative curvature. If $M$ is simply connected, then it is a well-known fact that there exists a unique geodesic between any pair of points $(x,x') \in \partial M \times \partial M$. The \textit{renormalized boundary distance} is defined as:
\[ D : \left| \begin{array}{l} \partial M \times \partial M \rightarrow \R \\ (x,x') \mapsto L(\alpha(x,x')) \end{array} \right., \]
where $L(\alpha(x,x'))$ denotes the renormalized length of the unique geodesic joining $x$ to $x'$. A preliminary version of our main theorem is the following result:
\begin{theo}
Assume $(M,g_1)$ and $(M,g_2)$ are two simply connected asymptotically hyperbolic surfaces of negative curvature. We suppose that for some choices $h_1$ and $h_2$ of conformal representatives in the conformal infinities of $g_1$ and $g_2$, the renormalized boundary distance agree for the two metrics, i.e. $D_1 = D_2$. Then, there exists a smooth diffeomorphism $\Phi : \overline{M} \rightarrow \overline{M}$ such that $\Phi^*g_2 = g_1$ on $M$ and $\Phi|_{\partial M} = \text{Id}$.
\end{theo}
In the terminology of \cite{ggsu}, such surfaces are called \textit{simple}: this definition naturally extends the notion of a simple manifold (compact manifold with boundary such that the exponential map is a diffeomorphism at each point) to the non-compact setting.

We are actually able to deal with the case of negatively-curved surfaces with topology. The natural object one has to consider this time is the \textit{renormalized marked boundary distance}. In this case, given two points $(x,x') \in \partial M \times \partial M$, there exists a unique geodesic in each homotopy class $[\gamma] \in \mathcal{P}_{x,x'}$ of curves joining $x$ to $x'$ ($\mathcal{P}_{x,x'}$ being the set of homotopy classes). We define
\[\mathcal{D} := \left\{ (x,x',[\gamma]), (x,x') \in \partial M \times \partial M \setminus \text{diag}, [\gamma] \in \mathcal{P}_{x,x'} \right\},\]
and introduce the renormalized marked boundary distance $D$ as:
\be \label{eq:ren} D : \left| \begin{array}{l} \mathcal{D} \rightarrow \R \\ (x,x',[\gamma]) \mapsto L(\alpha(x,x')) \end{array} \right., \ee
where $\alpha(x,x')$ is the unique geodesic in $[\gamma]$ joining $x$ to $x'$ and $L$ the renormalized length. Our main result is the following:

\begin{theo}
\label{th}
Assume $(M,g_1)$ and $(M,g_2)$ are two asymptotically hyperbolic surfaces of negative curvature. We suppose that for some choices $h_1$ and $h_2$ of conformal representatives in the conformal infinities of $g_1$ and $g_2$, the renormalized marked boundary distance agree for the two metrics, i.e. $D_1 = D_2$. Then, there exists a smooth diffeomorphism $\Phi : \overline{M} \rightarrow \overline{M}$ such that $\Phi^*g_2 = g_1$ on $M$ and $\Phi|_{\partial M} = \text{Id}$.
\end{theo}

This result can be seen as an analogue of \cite[Theorem 2]{gm} for the case of asymptotically hyperbolic surfaces. It is very likely that one can relax the assumption in Theorem \ref{th} so that only one of the two metrics has negative curvature (but still a hyperbolic trapped set). In the usual terminology, Theorem \ref{th} roughly says that an asymptotically hyperbolic surface with negative curvature is \textit{marked boundary distance rigid} among the class of surfaces having negative curvature. 

This result follows in spirit the ones proved independently by Otal \cite{jo} and Croke \cite{cc} establishing that two negatively-curved closed surfaces with same marked length spectrum are isometric. More recently, Guillarmou and Mazzucchelli \cite{gm} extended Otal's result to the case of two surfaces with strictly convex boundary without conjugate points and a trapped set of zero Liouville measure, one being of negative curvature. In both cases, the central object of interest is the \textit{Liouville current} $\eta$, which is the natural projection of the Liouville measure $\mu$ (initially defined on the unit tangent bundle $SM$) on the set of geodesics $\mathcal{G}$ of the manifold. Our arguments follow in principle the layout of proof of these articles, but we need to address new issues caused by the loss of the compactness assumption. The crucial step in our proof to deal with the infinite ends of the manifold is a version of Otal's lemma (see \cite[Lemma 8]{jo}) with a stability estimate (Proposition \ref{prop:ot2}). To the best of our knowledge, this bound had never been stated before in the literature.

As far as we know, this is also the first boundary rigidity result obtained in a non-compact setting. There is a long history of results regarding the boundary rigidity question on simple manifolds in the compact setting. We here mention the contributions of Gromov \cite{gr}, for regions of $\R^n$, the original paper of Michel \cite{rm} for subdomains of the open hemisphere and the Besson-Courtois-Gallot theorem \cite{bcg}, which implies the boundary rigidity for regions of $\HH^n$ (see also the survey of Croke \cite{cr2}). In the case of a manifold with trapping, the first general results where obtained by Guillarmou-Mazzucchelli \cite{gm} for surfaces, where the local boundary rigidity was established under suitable assumptions. Global boundary rigidity theorems have also recently been obtained by Stefanov-Uhlmann-Vasy \cite{suv} for simply connected non-positively curved manifolds with strictly convex boundary. Let us eventually mention that boundary rigidity questions appear naturally in the physics literature concerning the AdS/CFT duality and holography (see \cite{pr}, \cite{clms})

\subsection{Outline of the proof}

Section \ref{sect:geo} formally introduces the notion of renormalized length for a geodesic. We heavily rely on the cautious study made in \cite{ggsu} of the geodesic flow near the boundary at infinity. We also provide some examples of computations, and study the action of isometries on the renormalized length. In Section \ref{sect:liouv}, we recall the definition of the Liouville current $\eta$ on the space of geodesics of the universal cover $\widetilde{M}$ and prove that if the renormalized marked lengths agree, then the Liouville currents agree, juste like in the compact setting.

Section \ref{sect:dev} is devoted to the construction of an application of deviation $\kappa$. Like in \cite{jo}, we introduce \textit{the angle of deviation} $f$ between the two metrics on the universal cover $\widetilde{M}$. The idea is to make use of Gauss-Bonnet formula, in order to prove that this angle is the identity. This requires to introduce an \textit{average angle of deviation}. Since we are in a non-compact setting, technical issues arise from the fact that the volume is infinite. In particular, we need to consider this average (denoted by $\Theta_\eps$) on compact domains $\left\{\rho\geq\eps\right\}$ parametrized by $\eps$ and to study their limit as $\eps \rightarrow 0$.

Because of the possible existence of a \textit{trapped set}, we are unable to prove a priori that the averages $\Theta_\eps$ are $\mathcal{C}^1$ (or at least uniformly Lipschitz), which would truly simplify the proof. A cautious analysis of the derivative of the angle of deviation $f$ is needed to deal with these technical complications. Combined with a version of Otal's lemma with an estimate (see Proposition \ref{prop:ot2}), this allows to conclude that the average angle of deviation is the identity in the limit $\eps \rightarrow 0$, which itself implies that the angle of deviation $f$ is the identity.

We then conclude the proof by constructing a natural application $\Phi$ which is an isometry between $(M,g_1)$ and $(M,g_2)$. Eventually, a last difficulty comes from the fact that it is not immediate that the isometry obtained is $\mathcal{C}^\infty$ down to the boundary of $\overline{M}$.

If the reader is familiar with Otal's proof \cite{jo}, he will morally see the same features appear, but the novelty here is that we are able to deal with the asymptotic ends of the manifold. The price we have to pay is that this requires to compute tedious estimate in the limit $\eps \rightarrow 0$

\subsection{Acknowledgments} We thank Colin Guillarmou for suggesting this result and fruitful discussions.

\section{Geometric preliminaries}

\label{sect:geo}

This section is not specific to the two-dimensional case, so we state it in full generality. $(M,g)$ is only assumed to be an $n+1$-dimensional asymptotically hyperbolic manifold. In our setting, it will be more convenient to work on the unit cotangent bundle rather than on the unit tangent bundle, using the construction of Melrose \cite{me} of b-bundles.

\subsection{Geometry on the unit cotangent bundle}

\subsubsection{The b-cotangent bundle}

\label{sssect:b}

We define the unit cotangent bundle by
\be S^*M := \left\{(x,\xi) \in T^*M, x \in M, \xi \in T_x^*M, |\xi|^2_g = 1\right\}, \ee
and by $\pi_0 : S^*M \rightarrow M$ the projection on the base. The geodesic flow $(\varphi_t)_{t \in \R}$ is induced by the Hamiltonian vector field $X$, obtained from the Hamiltonian $H(x,\xi) = \frac{1}{2}|\xi|^2_g$, that is $\omega(X,\cdot)+dH=0$, where $\omega$ is the canonical symplectic form on $T^*M$. $\lambda$ will denote the canonical $1$-form such that $\omega = d\lambda$ and $\lambda(X)=1$. Note that $X$ is tangent to $S^*M$  insofar as a point $(x,\xi) \in T^*M$ flows by $X$ in the hypersurfaces $\left\{H=\text{cst}\right\}$. We will denote by $\flat : TM \rightarrow T^*M$ the Lagrange transform, that is the musical isomorphism between these two vector bundles and by $\sharp : T^*M \rightarrow TM$ its inverse. We stress that we will often drop the notation of these isomorphisms and identify (without mentioning it) a vector with its dual covector.

There exists a canonical splitting of $T(S^*M)$ according to:
\be T(S^*M) = \mathcal{H} \oplus \mathcal{V}, \ee
where $\mathcal{V} := \ke d\pi$ is the vertical bundle and $\mathcal{H} := \ke \mathcal{K}$ is the horizontal bundle. $\mathcal{K}$ is the connection map, defined such that, for $\zeta \in T(S^*M)$, $\mathcal{K}(\zeta) \in T_{\pi(\zeta)}M$ is the only vector such that the local geodesic $t \mapsto \gamma(t) \in SM$ starting from $(\pi(\zeta),\mathcal{K}(\zeta))$ satisfies $\dot{\gamma}(0) = \zeta^{\sharp}$ (see \cite{p} for a reference). The metric $g$ on $M$ induces a natural metric $G$ on $S^*M$, called the \textit{Sasaki metric} and defined by:
\be G(\xi,\zeta) := g(d\pi(\xi),d\pi(\zeta)) + g(\mathcal{K}(\xi),\mathcal{K}(\zeta)) \ee

Recall from \cite{me} that the \textit{b-tangent bundle} ${}^b T\overline{M}$ is defined to be the smooth vector bundle whose sections are vectors fields tangent to $\partial M$. Let $V$ be a smooth vector field on $M$. If $(\rho,y_1,...,y_n)$ denotes smooth local coordinates in a vicinity of $\partial M$, we can identify $V$ with the derivation
\[ V = a \partial_\rho + \sum_i b_i \partial_{y_i}, \]
for some smooth functions $a, b_i$. If $V$ vanishes on the boundary, then $a|_{\partial M}=0$, and we can write $a = \rho \alpha$ for some smooth function $\alpha$. In other words, in coordinates, $(\rho\partial_\rho,\partial_{y_i})$ is a basis for ${}^b T\overline{M}$ ($\rho \partial_\rho$ is a derivation down to the boundary $\left\{\rho=0\right\}$). Now, $\rho\partial_\rho$ is well defined on $\partial M$, independently of the choice of coordinates in a neighborhood of $\partial M$. Indeed, if $(\rho',y')$ denotes another choice of coordinates, then one can write $\rho' = \rho A(\rho,y), y'_i = Y_i(\rho,y)$ for some smooth functions such that $A(0,0)>0$, $J_{Y_i} \neq 0$ (where $J_{Y_i}$ denotes the Jacobian of $Y_i$) and one has
\[ \rho \partial_{\rho} = \left(1 + \dfrac{\rho}{A}\right) \rho' \partial_{\rho'} + \dfrac{\rho'}{A} \sum_j  \partial_\rho (Y_j) \partial_{y'_j}, \]
that is, both vector fields agree on the boundary $\left\{\rho'=0 \right\}$ as elements of ${}^{b} T \overline{M}$.

The \textit{b-cotangent bundle} ${}^{b} T^* \overline{M}$ is the vector bundle of linear forms on ${}^{b} T \overline{M}$. In coordinates, $(\rho^{-1} d\rho, dy_i)$ forms a basis of ${}^{b} T^* \overline{M}$ and $\rho^{-1} d\rho$ on $\partial M$ (the covector associated to $\rho \partial_\rho$) is independent of any choice of coordinates (and of the metric $g$). There are natural natural coordinates $(\rho,y,\xi =\xi_0 d\rho + \sum_{i} \eta_i d{y_i})$ on $T^*\overline{M}$ and we introduce on ${}^{b} T^* \overline{M}$ the smooth coordinates $(x,\xi) = (\rho,y,\overline{\xi_0},\eta)$, where $\xi_0=\overline{\xi_0} \rho^{-1}$, that is $\xi = \overline{\xi_0} \rho^{-1} d\rho + \sum_{i} \eta_i d{y_i}$. In particular, we see from the previous discussion that the function $\xi \mapsto \overline{\xi}_0$ on ${}^{b} T^* \overline{M}|_{\partial M}$ is intrinsic to the manifold, as well as the two subsets $\left\{\overline{\xi}_0 = \pm 1\right\}$ of ${}^{b} T^* \overline{M}|_{\partial M}$ (they do not depend on the choice of coordinate $(\rho,y)$, not even on the metric $g$).

As mentioned in the introduction, given a choice $h$ of conformal representative in the conformal infinity, there exist local coordinates $(\rho,y)$ in a vicinity of the boundary so that the metric has the form
\[ g = \dfrac{d \rho^2 + h_\rho}{\rho^2}, ~~~ \text{on } (0, \eps) \times \partial M, \]
Note that given $\xi = \overline{\xi_0} \rho^{-1} d\rho + \sum_{i} \eta_i d{y_i}$ in the b-cotangent bundle, one has:
\[ |\xi|^2_g = \overline{\xi}_0^2+\rho^2 |\eta|^2_{h_\rho}, \]
where, here, $h_\rho$ actually denotes the dual metric on $T^*\partial M$. We denote by:
\[ \overline{S^*M} = \left\{(x,\xi) \in {}^{b} T^* \overline{M}, |\xi|^2_g = 1\right\} \]
One has for $x \in \overline{M}$:
\[ \overline{S_x^*M} = \left\{(x,\xi) \in {}^{b} T^* \overline{M}, \overline{\xi}_0^2+\rho^2 |\eta|^2_{h_\rho} = 1\right\}\]
As a consequence, there is a splitting:
\[ \overline{S^*M} = S^*M \sqcup \partial_- S^*M \sqcup \partial_+ S^*M, \]
where $\partial_\pm S^*M = \left\{ (x,\xi), x \in \partial M, \overline{\xi_0}=\mp 1\right\}$ (which are independent of any choice). We see $\partial_- S^*M$ (resp. $\partial_+ S^*M$) as the \textit{incoming} (resp. \textit{outcoming}) boundary.

\begin{lemm}{\cite[Lemma 2.1]{ggsu}}
\label{lem:xbar}
There exists a smooth vector field $\overline{X}$ on $\overline{S^*M}$ which is transverse to the boundary $\partial \overline{S^*M} = \partial_- S^*M \sqcup \partial_+ S^*M$ and satisfies $X = \rho \overline{X}$ on $S^*M$. Moreover, in a vicinity of $\partial M$, one has $\overline{X} = \partial_\rho + \rho Y$, for some smooth vector field $Y$ on $\overline{S^*M}$.
\end{lemm}

The flow on $\overline{S^*M}$ induced by $\overline{X}$ will be denoted by $\overline{\varphi}_\tau$. For $z \in S^*M$, one has $\varphi_t(z) = \overline{\varphi}_\tau(z)$, for
\be \label{eq:xbar} t(\tau,z) = \int_0^\tau \dfrac{1}{\rho(\overline{\varphi}_s(z))} ds \ee

\subsubsection{Trapped set} \label{sssect:trapped}

The results of the following paragraph can be found in \cite[Section 2.1]{ggsu}. For $\eps > 0$ small enough, the compact surfaces $M_\eps := M \cap \left\{\rho\geq\eps\right\}$ are strictly convex with respect to the geodesic flow.

\begin{lemm}{\cite[Lemma 2.3]{ggsu}}
\label{lem:conv}
There exists $\eps > 0$ small enough so that for each $(x,\xi) \in S^*M$ with $\rho(x) < \eps$, $\xi = \xi_0 d\rho + \sum_{i=1}^{n-1} \xi_i dy_i$ and $\xi_0 \leq 0$, the flow trajectory $\varphi_t(x,\xi)$ converges to some point $z_+ \in \partial_+S^*M$ with rate $\mathcal{O}(e^{-t})$ as $t \rightarrow + \infty$ and $\rho(\varphi_t(x,\xi)) \leq \rho(x,\xi)$ for all $t \geq 0$. The same result holds with $\xi_0 \geq 0$ and negative time, with limit point $z_- \in \partial_-S^*M$.
\end{lemm}

We define the \textit{tails} $\Gamma_\pm$: they consist of the points in $S^*M$ which are respectively trapped in the past or in the future:
\be S^*M \setminus \Gamma_\mp := \left\{ z \in S^*M, \rho(\varphi_t(z))_{t \rightarrow \pm \infty} \rightarrow 0 \right\} \ee
The \textit{trapped set} $K$ is defined by:
\be K := \Gamma_+ \cap \Gamma_- \ee
In particular, in negative curvature, the trapped set has zero Liouville measure. We can define the exit and enter maps $B_\pm : S^*M \setminus \Gamma_\mp \rightarrow \partial_\pm S^*M$ such that \be B_\pm(z) := \lim_{t \rightarrow \pm \infty} \varphi_t(z)\ee
These are smooth, well-defined map and they extend smoothly to $\overline{S^*M} \setminus \overline{\Gamma_\mp}$, where $\overline{\Gamma_\mp}$ is the closure of $\Gamma_\mp$ in $\overline{S^*M}$ (see \cite[Corollary 2.5]{ggsu}). There also exist smooth functions $\tau_\pm : \overline{S^*M} \setminus \overline{\Gamma_\mp} \rightarrow \R_\pm$ defined such that:
\be \overline{\varphi}_{\tau_\pm(z)}(z) = B_\pm(z) \in \partial_\pm S^*M \ee
Using the vector field $\overline{X}$, another way of describing the sets $\overline{\Gamma_\pm}$ is
\be \overline{\Gamma_\pm} = \left\{z \in \overline{S^*M}, \tau_\mp(z) = \pm \infty \right\} \ee
The \textit{scattering map} is the smooth map $\sigma : \partial_- S^*M \setminus \overline{\Gamma_-} \rightarrow \partial_+ S^*M \setminus \overline{\Gamma_+}$ defined by:
\be \label{eq:scat} \sigma(z) := B_+(z) = \overline{\varphi}_{\tau_+(z)}(z) \ee

\subsubsection{Hyperbolic splitting in negative curvature}

We assume in this section that $(M,g)$ is two-dimensional and has negative curvature $\kappa < 0$. Since the curvature at infinity converges towards $-1$, we know that $\kappa$ is pinched between two constants $-k_0^2 \leq \kappa < -k_1^2 < 0$. It is a classical fact that the geodesic flow on such a surface is Anosov (see \cite{pe}, \cite{wk}) in the sense that there exists some constants $C > 0$ and $\nu > 0$ (depending on the metric $g$) such that for all $z = (x,\xi) \in S^*M$, there is a continuous flow-invariant splitting
\be \label{eq:split} T_z(S^*M) = \R X(z) \oplus E_u(z) \oplus E_s(z), \ee
where $E_s(z)$ (resp. $E_u(z)$) is the \textit{stable} (resp. \textit{unstable}) vector space in $z$, which satisfy
\be \begin{array}{c} |d\varphi_t(z) \cdot \xi|_{\varphi_t(z)} \leq C e^{-\nu t} |\xi|_{z}, ~~ \forall t > 0, \xi \in E_s(z) \\
|d\varphi_t(z) \cdot \xi|_{\varphi_t(z)} \leq C e^{-\nu |t|} |\xi|_{z}, ~~ \forall t < 0, \xi \in E_u(z)\end{array} \ee
The norm, here, is given in terms of the Sasaki metric. The bundles $z \mapsto E_u(z),E_s(z)$ are (Hölder) continuous everywhere on $S^*M$. Moreover, the differential of the geodesic flow is governed uniformly by an exponential growth (see \cite[Chapter 3]{ro}) in the sense that there exists (other) constants $C, k > 0$ such that:
\be \label{eq:growth} |d\varphi_t(z) \cdot \xi|_{\varphi_t(z)} \leq C e^{k t} |\xi|_{z}, ~~ \forall t > 0, \forall \xi \in T_z(S^*M) \ee

\begin{wrapfigure}[11]{r}{5cm}
\label{fig:esc}
\includegraphics[width=7cm]{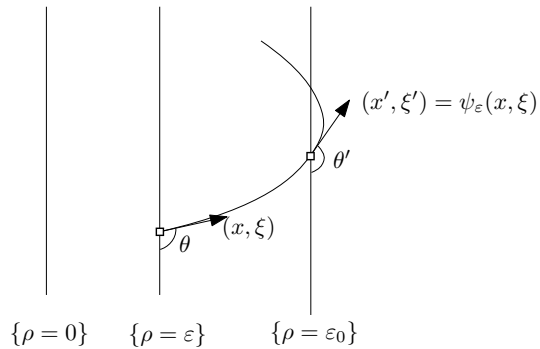}
\caption{The diffeomorphism $\psi_\eps$ in the proof of Lemma \ref{lem:escape}}
\end{wrapfigure}

Let us  now fix $\eps > 0$ small enough and consider $M_\eps := M \cap \left\{\rho\geq\eps\right\}$. Like in \cite{gu}, we define the \textit{non-escaping mass function} $V_\eps(T)$ of the domain $M_\eps$ by $V_\eps(T) := \mu\left(\left\{z \in S^*M_\eps, \forall s \in [0,T], \varphi_s(z) \in S^*M_\eps\right\}\right)$.
Since the trapping set is hyperbolic, there exists a constant $Q < 0$ such that $Q := \limsup_{T \rightarrow +\infty} \dfrac{\ln(V_\eps(T))}{T}$. Note that this constant is independent of $\eps$ (see \cite[Proposition 2.4]{gu}). In the rest of this paragraph, we fix some $\eps_0 > 0$ small enough. For $0 < \eps < \eps_0$, we want to link explicitly the decay of the non-escaping mass function $V_\eps$ to $V_{\eps_0}$.

\begin{lemm}
\label{lem:escape}
Let $\delta \in (Q,0)$. There exists a constant $C > 0$ and an integer $\wt{N} \in \N^*$, such that for all $T \geq -\wt{N} \ln(\eps)$,
\[ V_\eps(T) \leq C\eps^{-(1+4\delta)}e^{-\delta T} \]
\end{lemm}

\begin{proof}

For $(x,\xi) \notin \Gamma_-$ we denote by $l_{\eps,+}(x,\xi)$ the exit time of the manifold $M_\eps$, that is the maximum time such that: $\forall t \in [0,l_{\eps,+}(x,\xi)], \varphi_t(x,\xi) \in S^*M_\eps$. By Santalò's formula, we can express $V_\eps(T)$ as:
\[ V_\eps(T) = \int_{\partial_-S^*M_\eps} \left(l_{\eps,+}(x,\xi) - T\right)_+ d\mu_{\nu,\eps} \]
where $x_+ = \sup(x,0)$, $d\mu_{\nu,\eps}(x,\xi) = |g(\xi,\nu)|i^*_{\partial S^*M_\eps}(d\mu)$\footnote{The metric $g$ here actually denotes the dual metric to $g$ which is usually written $g^{-1}$. As mentioned in the introduction, we do not bother with such notations in order to keep the reading affordable.}, $\nu$ is the normal unit outward covector to the boundary, $i^*_{\partial S^*M_\eps}(d\mu)$ is the restriction of the Liouville measure to the boundary (the measure induced by the Sasaki metric restricted to $\partial S^*M_\eps$). There exists a maximum time $T^*_\eps$, such that given any $(x,\xi) \in \partial_+S^*M_{\eps_0}$, $\varphi_{T_\eps}(x,\xi)$ has exited the manifold $M_\eps$. One can bound this time $T^*_\eps$ by $\ln(C\eps_0/\eps)$, where $C > 0$ is some constant independent of $(x,\xi)$ and $\eps$ (see the proof of \cite[Lemma 2.3]{ggsu}). We introduce $T_\eps:=-2\ln(\eps) > T^*_\eps$ for $\eps$ small enough. As a consequence, for $T \geq 2T_\eps$, one has:
\[ V_\eps(T) \leq \int_{\partial_-S^*M_\eps \cap D_\eps} \left(l_{\eps_0,+}(\psi_\eps(x,\xi))-(T-2T_\eps)\right)_+ d\mu_{\nu,\eps}, \]
where $\psi_\eps^{-1} : \partial_-S^*M_{\eps_0} \rightarrow \psi_\eps^{-1}(\partial_-S^*M_{\eps_0}) =: D_\eps \subset \partial_-S^*M_{\eps}$ is the diffeomorphism which flows backwards (by $\overline{\varphi}_\tau$) a point $(x,\xi) \in \partial_-S^*M_{\eps_0}$ to the boundary $\partial_-S^*M_{\eps}$ (see Figure (\ref{fig:esc})).

The dependence of $\psi_\eps^{-1}$ on $\eps$ is smooth down to $\eps=0$: this follows from the implicit function theorem. In the local product coordinates $(\rho,y)$, one can write $d\mu_{\nu,\eps} = 1/\eps \sin(\theta) h(\eps,y) dy d\theta$, where $[0,\pi] \ni \theta \mapsto \xi(\theta)$ parametrizes the cosphere fiber, $h$ is a smooth non-vanishing function down to $\eps=0$. The point $(x,\xi)$ corresponds to $(y,\theta)$ in these coordinates and we write $(y',\theta') = \psi_\eps(y,\theta)$. If $T$ is large enough, for the integrand not to vanish, one has to require that the angle $\theta'(\psi_\eps(y,\theta))$ is uniformly contained in a compact interval of $]0,\pi[$. In other words, if we fix some constant $c > 0$, there exists an integer $\wt{N} \geq 2$ large enough (independent of $\eps$) such that for $T \geq -\wt{N}\ln(\eps)$, if $\theta'(\psi_\eps(y,\theta)) \in [0,c] \cup [\pi-c,\pi]$, it will satisfy $\left(l_{\eps_0,+}(\psi_\eps(y,\theta))-(T-2T_\eps)\right)_+ =0$. We can now make a change of variable in the previous integral by setting $(y',\theta') = \psi_\eps(y,\theta)$. Since the dependence of $\psi_\eps^{-1}$ is smooth in $\eps$ (down to $\eps=0$) and $[0,\eps_0] \times \left\{\rho=\eps_0\right\}$ is compact, $|\det(\psi_\eps^{-1}(y',\theta'))|$ is bounded independently of $(y',\theta')$ and $\eps$. We get for $T \geq -\wt{N} \ln(\eps)$:
\[ \begin{split}  \int_{\partial_-S^*M_\eps \cap D_\eps} & \left(l_{{\eps_0,+}}(\psi_\eps(x,\xi))-(T-2T_\eps)\right)_+ d\mu_{\nu,\eps} \\ 
& = \int_{\partial_-S^*M_\eps \cap D_\eps} \left(l_{{\eps_0,+}}(\psi_\eps(y,\theta))-(T-2T_\eps)\right)_+ \sin(\theta) h(\eps,y) \dfrac{dy d\theta}{\eps} \\
& = \int_{\partial_-S^*M_{\eps_0}} \left(l_{\eps_0,+}(y',\theta')-(T-2T_\eps)\right)_+ \sin\left(\theta (\psi_\eps^{-1}(y',\theta'))\right) \\
& \hspace{3cm} h\left(\eps,y(\psi_\eps^{-1}(y',\theta'))\right) |\det(\psi_\eps^{-1}(y',\theta'))| \dfrac{d\theta' dy'}{\eps}  \\
& \leq C \int_{\partial_-S^*M_{\eps_0}} \left(l_{\eps_0,+}(y',\theta')-(T-2T_\eps)\right)_+ \dfrac{d\theta' dy'}{\eps} \\
& \leq  C \eps^{-1} \int_{\partial_-S^*M_{\eps_0,+}} \left(l_{\eps_0}(y',\theta')-(T-2T_\eps)\right)_+ h(\eps_0,y) \sin(\theta')\dfrac{d\theta' dy'}{\eps_0}  \\
& \leq C \eps^{-1} V_{\eps_0}(T-2T_\eps), \end{split} \]
for some constant $C > 0$ (which may be different from one line to another) and where the penultimate inequality follows from the uniform bound on the angle (i.e. $\sin(\theta') \in [\sin(c),1]$). But we know that for any $\delta \in (Q,0)$, there exists an (other) constant $C > 0$ such that for all $T \geq 0$, $V_{\eps_0}(T) \leq Ce^{-\delta T}$. Thus, for $T \geq -\wt{N} \ln(\eps)$
\[ V_\eps(T) \leq C\eps^{-1}e^{-\delta(T-2T_\eps)}  \leq C\eps^{-(1+4\delta)}e^{-\delta T}  \]
\end{proof}

\subsection{The renormalized length}

\subsubsection{Definition} \label{sssect:def} This paragraph follows the definition of \cite[Section 4.1]{ggsu}. Given a smooth compact curve $\gamma$ in $M$, we denote by $l(\gamma)$ its length with respect to the metric $g$. Let $\alpha(x,x')$ denote a geodesic in $M$ joining two points at infinity $x,x' \in \partial M$ (we assume that $x \neq x'$). For the sake of simplicity, we will only write $\alpha$ in this paragraph, instead of $\alpha(x,x')$.

\begin{lemm}
The renormalized length of the geodesic $\alpha$ is the real number defined by:
\be L(\alpha) := \lim_{\eps \rightarrow 0} l(\alpha \cap \left\{\rho \geq \eps\right\}) + 2\ln(\eps) \ee
\end{lemm}

\begin{proof}[Proof]
Let $z=(x,\xi) \in \partial_- S^*M$ be the point on the boundary generating $\alpha$. Note that since $\overline{X}$ is transverse to the boundary, $s \mapsto \rho(\overline{\varphi}_s(z))$ is a local diffeomorphism for $s$ in a vicinity of $0$ or $\tau_+(z)$. Let us consider $\delta > \eps > 0$, small enough. According to the expression (\ref{eq:xbar}), one has:
\[ l(\alpha \cap \left\{\rho \geq \eps\right\}) = \int_{\tau^1_\eps}^{\tau^2_\eps} \dfrac{1}{\rho(\overline{\varphi}_s(z))} ds, \]
where $\tau^1_\eps < \tau^2_\eps$ are defined to be the two unique times in $[0,\tau_+(z)]$ such that $\rho(\overline{\varphi}_{\tau^i_\eps}(z))=\eps$ (note that they are well-defined by the convexity property stated in Lemma \ref{lem:conv}). Splitting the integral, and using the smooth change of variable $u = \rho(\overline{\varphi}_s(z))$ on the two extremal intervals, we obtain:
\[ \begin{split} l(\alpha \cap \left\{\rho \geq \eps\right\}) & = \int_{\tau^1_\eps}^\delta \dfrac{ds}{\rho(\overline{\varphi}_s(z))} + \int_{\delta}^{\tau_+(z)-\delta} \dfrac{ds}{\rho(\overline{\varphi}_s(z))} + \int_{\tau_+(z)-\delta}^{\tau^2_\eps} \dfrac{ds}{\rho(\overline{\varphi}_s(z))}   \\
& = \int_{\eps}^{u(\delta)} \left(\dfrac{1}{u} + \mathcal{O}(1)\right) du + \int_{\eps}^{u(\tau_+(z)-\delta)} \left(\dfrac{1}{u} + \mathcal{O}(1)\right) du + \text{cst} \\
& = - 2\ln(\eps) + \text{cst}(\eps), \end{split} \]
where we used the fact that $\partial_s \left(\rho(\overline{\varphi}_s(z))\right)|_{s=0} = 1$, according to Lemma \ref{lem:xbar}. Here, $\delta$ is chosen independent of $\eps$ and $\text{cst}(\eps)$ admits a finite limite as $\eps \rightarrow 0$.
\end{proof}

In other words, we have $l(\alpha \cap \left\{\rho \geq \eps\right\}) = -2\ln(\eps) + L(\alpha) + o(1)$, that is $L(\alpha)$ is the first finite term in the asymptotic expansion of the length of the geodesic segment $\alpha_\eps:=\alpha \cap \left\{\rho \geq \eps\right\}$ as $\eps \rightarrow 0$, as explained in the introduction. In the following, we will indifferently write $L(\alpha)$, $\alpha$ being a geodesic joining two points of the boundary at infinity, or $L(z)$, for $z \in \partial_-S^*M \setminus \overline{\Gamma}_-$ generating the geodesic $\alpha$.

\begin{rema}
In \cite{ggsu}, another definition of the renormalized length is provided, using the X-ray transform defined on functions on $M$ (see \cite[Section 3.2]{ggsu}):
\be L(\alpha) = \left.\left(I_0(\rho^\lambda)(z)-2/\lambda\right)\right|_{\lambda=0}, \ee
where $z \in \partial_-SM \setminus \overline{\Gamma}_-$ is the point generating $\alpha$. Actually, $\lambda \mapsto I_0(\rho^\lambda)(z)$ is a meromorphic function on $\text{Re}(z) > -1$ with only a simple pole at $\lambda=0$ and $2$ for residue.
\end{rema}

Note that there is \textit{a priori} no canonical choice of the renormalized length $L$ insofar as it depends on the choice of the boundary defining function $\rho$. Thus, it is interesting to see how $L$ is modified when $\rho$ changes. One can actually prove that if $\hat{\rho} = e^\omega \rho$ is another choice, then (see \cite[Equation (4.2)]{ggsu}):
\[ \label{eq:conf} \hat{L}(z) - L(z) = \omega(\pi(z))+\omega(\pi(\sigma(z))) \]

\begin{rema} \label{rem:psi} As a consequence, if two defining functions induce the same representative for the conformal infinity, then they induce the same renormalized lengths. Thus, if $\psi : \overline{M} \rightarrow \overline{M}$ is a diffeomorphism which preserves the boundary, $\rho \circ \psi$ and $\rho$ induce the same representative for the conformal infinity, so $L_g = L_{\psi^*g}$, where both renormalized lengths are computed with respect to $\rho$. \end{rema}

\subsubsection{An example : the hyperbolic disk}

Let us consider the hyperbolic disk $(\D, \frac{4|dz|^2}{(1-|z|^2)^2})$. The set of geodesics on $\D$ can be naturally identified with $\partial \D \times \partial \D \setminus \text{diag}$ insofar as there exists a unique geodesic joining to points on the ideal boundary. There is a natural choice for the boundary defining function which is given by $\rho(z) := \frac{1}{2}(1-|z|^2)$.

\begin{prop}
Let $\xi, \zeta \in \partial \D$ be two points on the boundary. Then:
\be \label{eq:lendisk} L(\xi,\zeta) = 2 \ln(|\xi-\zeta|) \ee
\end{prop}

\begin{proof}[Proof]
We denote by $\alpha$ the geodesic joining $\xi$ to $\zeta$. For $\eps > 0$, we denote by $p_\eps$ and $q_\eps$ the points of intersection of $\alpha$ with $\left\{\rho=\eps\right\}$ in a respective vicinity of $\xi$ and $\zeta$. We have:
\[ d(p_\eps,q_\eps) = \ln \left( \dfrac{|\xi-q_\eps||\zeta-p_\eps|}{|\xi-p_\eps||\zeta-q_\eps|} \right) = 2 \ln(|\xi-q_\eps|) - 2\ln(|\xi-p_\eps|),\]
by symmetry. As $\eps \rightarrow 0$, $|\xi-q_\eps| \rightarrow |\xi-\zeta|$ and, using elementary arguments of geometry, one can prove that $|\xi-p_\eps| = \eps(1+o(1))$. Thus:
\[ \begin{split} d(p_\eps,q_\eps) & = 2 \ln(|\xi-\zeta|) -2\ln(\eps) - 2\ln(1+o(1))\end{split} \]
\end{proof}

\begin{rema}
In the model of the hyperbolic plane $(\HH, \frac{dx^2+dy^2}{y^2})$, if one takes the boundary defining function $\rho(x,y) = y$, then given two points $\xi, \zeta$ on the real line, one can check that:
\[ L(\xi,\zeta) = 2\ln(|\xi-\zeta|) \]
\end{rema}

We see in particular that the renormalized length is not a proper "length" according to the usual terminology insofar as it can be negative, and we even have $L(\xi,\zeta) \rightarrow - \infty$ as $\xi \rightarrow \zeta$. This is not specific to the hyperbolic disk and can be proved in the general frame. Moreover, we see from the expression (\ref{eq:lendisk}) that the renormalized length is not invariant by the isometries of the disk
\[ \gamma : z \mapsto e^{i\theta} \dfrac{z+c}{cz+1}, \]
if $c \neq 0$, but:
\[ \begin{split} L(\gamma(\xi),\gamma(\zeta)) & = 2\ln(|\gamma(\xi)-\gamma(\zeta)|)\\
& = 2 \ln(|\xi-\zeta||\gamma'(\xi)|^{1/2}|\gamma'(\zeta)|^{1/2}) \\
& = L(\xi,\zeta) + \ln(|\gamma'(\xi)||\gamma'(\zeta)|) \end{split} \]

\subsubsection{Action of isometries on the renormalized length} 

\label{sssect:isom}

Recall that a point on the boundary $\partial M$ is identified with the set of geodesics that asymptotically converge towards this point: two geodesics $\alpha_1$ and $\alpha_2$ induce the same point on the ideal boundary if there exists a constant $C>0$ (depending both on $\alpha_1$ and $\alpha_2$) such that for all $t \geq 0, d(\alpha_1(t), \alpha_2(t)) \leq C$. If $\gamma$ is an isometry on $M$, then one has $d(\gamma \circ \alpha_1 (t), \gamma \circ \alpha_2 (t)) \leq C$ for $t \geq 0$, which means that $\gamma \circ \alpha_1$ and $\gamma \circ \alpha_2$ represent the same point on the ideal boundary. In other words, the action of $\gamma$ on $M$ can be naturally extended to $\partial M$ and $\gamma : \overline{M} \rightarrow \overline{M}$ is at least continuously differentiable (see \cite[Proposition 2.11]{aw} for instance).

\begin{lemm}
Let $\alpha$ be a geodesic joining two points $x, x' \in \partial M$. We have:
\[ L(\gamma \circ \alpha) = L(\alpha) + \frac{1}{n} \ln(|d\gamma_x| |d\gamma_{x'}|), \]
where $|d\gamma_x|$ is the Jacobian of $\gamma|_{\partial M}$ in $x$ with respect to the metric $h$, $n+1$ being the dimension of $M$.
\end{lemm}

\begin{proof}[Proof]
We denote by $z=(x,\xi)$ the point in $\partial_- S^*M$ generating $\alpha$. We fix some $\delta > 0$: in the following, we will only study the half-line $\widetilde{\alpha} := \alpha \cap \left\{\rho\leq\delta\right\}$ (the other part of $\alpha$ can be studied in the same exact fashion). Let $x_\eps := \widetilde{\alpha} \cap \left\{\rho=\eps\right\}$ and $\alpha_\eps := \widetilde{\alpha} \cap \left\{\rho\geq\eps\right\}$. We define $\eps' := \rho(\gamma(x_\eps))$. We have:
\[ \label{eq:eps2} \begin{split} l(\alpha_\eps) + \ln(\eps) = \left(l(\gamma(\alpha_\eps)) + \ln(\eps')\right) - \ln(\eps'/\eps) \end{split} \]
As $\eps \rightarrow 0$, the left-hand side converges to $L(\widetilde{\alpha})$ whereas the term between parenthesis on the right-hand side goes to $L(\gamma(\widetilde{\alpha}))$, so all is left to compute is the limit of $\eps'/\eps$ as $\eps \rightarrow 0$. We write $\eps' = \rho(\gamma(\pi_0(\overline{\varphi}_{\tau_\eps}(z))))$, where $\tau_\eps$ is defined to be the unique time such that $\rho(\overline{\varphi}_{\tau_\eps}(z)) = \eps$. By the implicit function theorem, $\eps \mapsto \tau_\eps$ is a smooth function of $\eps$ and it satisfies: $\rho(\overline{\varphi_{\tau_\eps}}(z)) = \eps = \tau_{\eps} + \mathcal{O}(\tau_\eps^2)$. Thus $\left.\dfrac{d\tau_\eps}{d\eps}\right|_{\eps=0} = 1$ and:
\[ \begin{split} \lim_{\eps \rightarrow 0} \eps'/\eps & = d\rho_{\gamma(x)}\left(d\gamma_x\left(d{\pi_0}_z\left(\left.\dfrac{d\tau_\eps}{d\eps}\right|_{\eps=0} \dfrac{\partial\overline{\varphi}_{\tau_\eps}}{\partial \tau_\eps}(z)\right)\right)\right) \\
& = d\rho_{\gamma(x)}\left(d\gamma_x(d{\pi_0}_z(\overline{X}(z)))\right) \\
& = d\rho_{\gamma(x)}\left(d\gamma_x(\partial_\rho(x))\right)  \end{split}\]
Remark that $d\gamma_x(\partial_\rho(x)) = \lambda(x) \partial_\rho(\gamma(x))$ for some real number $\lambda$ depending on $x$, since $\gamma$ sends geodesics on geodesics. If $\eta_1, ..., \eta_{n} \in T_x(\partial M)$ is an orthonormal basis for the metric $h$, one can prove that $h(d\gamma_x(\eta_i),d\gamma_x(\eta_j)) = \lambda^2(x) \delta_{ij}$ by using the fact that $\gamma^*g =g$. As a consequence, the Jacobian of $\gamma|_{\partial M}$ in $x$ with respect to the metric $h$ is $\lambda^{n}(x)$. Thus:
\[ \lim_{\eps \rightarrow 0} \eps'/\eps = |d\gamma_x|^{\frac{1}{n}}\]
Replacing this in (\ref{eq:eps2}), and adding the other part of the geodesic, we find the sought result.
\end{proof}

\subsection{Homotopy on an asymptotically hyperbolic manifold}

$\overline{M}$ is a smooth compact manifold with boundary. Let us denote by $\pi : \widetilde{\overline{M}} \rightarrow \overline{M}$ its universal cover, which is a smooth non-compact surface with boundary $\partial \widetilde{\overline{M}}$: its boundary is a countable union of connected components which project down to $\partial M$ through $\pi$. The interior of $\widetilde{\overline{M}}$, that is $\widetilde{\overline{M}} \setminus \partial \widetilde{\overline{M}}$ is actually $\widetilde{M}$, the universal cover of $M$. In the sequel, we will rather use the notation $\partial \widetilde{M}$ to denote $\partial \widetilde{\overline{M}}$, which we see as the ideal boundary for the asymptotic manifold $\widetilde{M}$. The metric $g$ can be pulled back to a metric $\widetilde{g} = \pi^* g$ on $\widetilde{M}$ so that $\pi$ is a local isometry. In the same fashion, the boundary defining function $\rho$ can be pulled back via $\pi$ on $\widetilde{\overline{M}}$ and $(\widetilde{M},\widetilde{g})$ thus almost has the structure of an asymptotically hyperbolic manifold, in the sense that for each point $p \in \partial \widetilde{M}$, one can find coordinate charts so that the metric $\widetilde{g}$ has the form (\ref{eq:g}).

The unit cotangent bundle $S^*\widetilde{M}$ of $\widetilde{M}$ is also a cover of $S^*M$ (it is not its universal cover though) and we will still denote by $\pi : S^*\widetilde{M} \rightarrow S^*M$ the covering map. The geodesic flow $\varphi_t$ on $SM$ lifts to $\widetilde{\varphi}_t$ on $\widetilde{M}$, or equivalently, the vector field $X$ lifts to $\widetilde{X}$ on $S^*\wt{M}$, which is nothing but the geodesic vector field induced by the metric $\widetilde{g}$. Note that $\overline{X} = X/\rho$ also lifts to $S^*M$, as the vector field $\widetilde{\overline{X}} = \widetilde{X}/\widetilde{\rho}$ which is smooth down to $\partial \widetilde{M}$.

Given two points $(x,x') \in \partial M \times \partial M$, there exists in each class of homotopy $[\gamma]$ joining these two points a unique geodesic $\alpha$ (seen as an unparametrized curve). This geodesic lifts by the isometry $\pi$ to a geodesic in $\widetilde{M}$ (for the metric $\widetilde{g}$) joining two pre-images of $x$ and $x'$. Now, the converse is also true: given two pre-images of $x$ and $x'$ on the boundary of $\widetilde{M}$, there exists a unique geodesic $\widetilde{\alpha}$ joining them and it projects down on $M$ to a geodesic joining $x$ to $x'$ in a certain homotopy class.

One can check that by construction, the renormalized length of the geodesic $\widetilde{\alpha}$ between two boundary points on $\widetilde{M}$ is equal to that of its projection $\alpha = \pi(\widetilde{\alpha})$ on $M$. In other words, the knowledge of the renormalized marked boundary distance on the boundary of $M$ is equivalent to the knowledge of the renormalized boundary distance function on $\partial \widetilde{M} \times \partial \widetilde{M}$.

\section{The Liouville current}

Like in the previous paragraphs, we denote by $\widetilde{M}$ the universal cover of $M$. It is a topological disk on which we fix an orientation. As explained previously, all the objects ($g, \rho, X$, ...) lift to $\wt{M}$ and their corresponding object in the universal cover is invariant by the action of the fundamental group $\pi_1(M)$. Since we will only work on $\wt{M}$ in the following, for the reader's convenience, we will often drop the notation $\wt{\cdot}$ when the context is clear, except for the universal cover itself $\wt{M}$. We define
\[ \mathcal{G} := (\partial \widetilde{M} \times \partial \widetilde{M}) \setminus \text{diag}, \]
which can be naturally identified with the set of untrapped geodesics (neither in the future nor the past) on $\wt{M}$ insofar as there exists a unique geodesic joining two boundary points. $\mathcal{M}$ is the set of Borel measures on $\mathcal{G}$ which are invariant by the flip.

\label{sect:liouv}

\subsection{The Liouville current in coordinates}

We define the diffeomorphism $\psi : \mathcal{G} \times \R \rightarrow S^*\widetilde{M} \setminus (\Gamma_- \cup \Gamma_+)$ according to the expression:
\[ \psi(x,x',t) := \alpha(t), \]
where $\alpha$ is the unique geodesic joining $x$ to $x'$ parametrized in the following way: if $z = (x,\xi) \in \partial_- S^* \widetilde{M}$ denotes the point generating $\alpha$, then we parametrize the geodesic by $\alpha(t) = \varphi_t(m(z))$, where $m(z) = \overline{\varphi}_{\tau_+(z)/2}(z)$ is the middle point (this is a smooth map according to Section \ref{sssect:trapped}).

By construction, we have $\psi_* \partial_t = X$. The Liouville volume form $\mu$ is pulled back by via $\psi$ on a volume form $\omega= f dx \wedge dx' \wedge dt$ (for some smooth and non-vanishing function $f$) on $\mathcal{G} \times \R$ which satisfies:
\[ i_{\partial_t} \omega = \psi^* \left( i_{X} (\lambda \wedge d\lambda) \right) = \psi^*(d \lambda) \]
Since $i_{X}(d\lambda) = 0$, we also have:
\[ \mathcal{L}_{\partial_t}  \psi^*(d \lambda) = d(i_{\partial_t} \psi^*(d \lambda)) = d(\psi^*(i_X(d\lambda))) = 0 \]
As a consequence, the function $f$ does not depend on $t$ and we can write
\[ d\left|\psi^* \mu\right| = d \eta \otimes dt, \]
for some measure $\eta \in \mathcal{M}$ which we call the Liouville current on $\mathcal{G}$. 


Let $\gamma(t) := \pi_0(\alpha(t))$. We define
\be \label{eq:v} V := \left\{ (\tau, \theta) \in \R \times (0, \pi), (\gamma(\tau), R_\theta \dot{\gamma}(\tau)) \notin \Gamma_- \cup \Gamma_+ \right\}, \ee
where $R_\theta$ is the rotation by a positive angle $\theta$ in the fibers of $S^* \widetilde{M}$. We denote by $\mathcal{F}(x,x') \subset \mathcal{G}$ the open subsets of points $(y,y') \in \mathcal{G}$ such that the geodesic joining $y$ to $y'$ has a transverse and positive (with respect to the orientation) intersection with the geodesic $\alpha$ in $\widetilde{M}$. Now, consider the diffeomorphism $\phi : V \mapsto \mathcal{F}(x,x')$ defined by $\phi(\tau,\theta) = (y,y')$, the two points in $\partial \widetilde{M}$ such that the geodesic connecting them passes through the point $(\gamma(\tau), R_\theta \dot{\gamma}(\tau)) \in S^*\widetilde{M}$. The following lemma is a well-known fact (see \cite[Lemma 3.1]{gm} for instance) and we do not provide its proof.

\begin{lemm}
\[ \phi^* \eta = \sin(\theta) d \theta d\tau \]
\end{lemm}

\begin{rema} \label{rem:rem} In negative curvature, the tails $\Gamma_- \cup \Gamma_+$ have zero Liouville measure. This implies that the set ${}^c V \subset \R \times (0,\pi)$ has zero measure in $\R \times (0,\pi)$ (for the measure $\sin(\theta) d \theta d\tau$). In particular, we will not have to bother with trapped geodesics in the computations of the integrals of Section \ref{ssect:av}. \end{rema}

From the previous expression in coordinates, we recover the classical formula for $(x,x') \in \widetilde{M} \times \widetilde{M}$ (see \cite{jo}):
\be \label{eq:d} \eta\left(\mathcal{F}(x,x')\right) = \int_0^\pi \int_0^{d(x,x')} \sin(\theta) d\theta d\tau = 2 d(x,x'), \ee
where $d(\cdot,\cdot)$ denotes the Riemannian distance between two points. Let us now consider $(x,x') \in \partial \widetilde{M} \times \partial \widetilde{M}$. For $\eps > 0$ small enough, we denote by $x_\eps$ and $x'_\eps$ the two intersections of $\alpha$ (the geodesic joining $x$ to $x'$) with $\left\{ \rho = \eps \right\}$ in a respective neighborhood of $x$ and $x'$. We have:
\[ \begin{split} \eta(\mathcal{F}(x_\eps,x'_\eps)) + 4 \ln \eps & = 2 \left( d(x_\eps,x'_\eps) + 2 \ln \eps\right) \\ &= 2 \left( l(\alpha \cap \left\{\rho > \eps \right\}) + 2\ln \eps \right) \\ & \rightarrow_{\eps \rightarrow 0} 2 L(\alpha) \end{split} \]

\subsection{Liouville current and boundary distance}

Let $g_1$ and $g_2$ be two negatively-curved metrics such that their renormalized length agree. We denote by $\eta_1$ and $\eta_2$ their respective Liouville current.

\begin{lemm}
\label{lem:agree}
$\eta_1=\eta_2$
\end{lemm}

\begin{proof}[Proof]
We recall that $\partial \widetilde{M}$ is a countable union of real lines embedded in the circle $\mathbb{S}^1$. The topology on $\partial \widetilde{M}$ is that naturally induced by the topology on $\mathbb{S}^1$. It is sufficient to prove the result for $\Omega = E \times F$, where $E$ is an interval with extremal points $x_1, x_2$ and $F$ an interval with extremal points $x_3,x_4$, such that $\overline{E} \cap \overline{F} = \emptyset$. Indeed, if $\overline{E} \cap \overline{F} \neq \emptyset$, then by the following computation and the fact that $L(\xi,\zeta) \rightarrow_{\xi \rightarrow \zeta} - \infty$, one gets that $\eta_1(\Omega) = \eta_2(\Omega)=+\infty$. We actually prove:

\begin{lemm}
\label{lem:long}
For $\Omega = E\times F$, a product of two disjoint intervals:
\be \eta(\Omega) = L(x_1,x_3) + L(x_2,x_4) - L(x_2,x_3) - L(x_1,x_4)\ee
\end{lemm}

Given some $\eps > 0$, we introduce the four horospheres $H_i(\eps), i \in \left\{1,...,4\right\}$ such that $H_i(\eps)$ intercepts $x_i$ and the point defined as the intersection of the geodesic $\alpha(x_i,x_{i+2})$ ($i+2$ is taken modulo $4$) with $\left\{\rho = \eps \right\}$ in a vicinity of $x_i$.

\begin{figure}[h!]
\begin{center}

\includegraphics[scale=0.6]{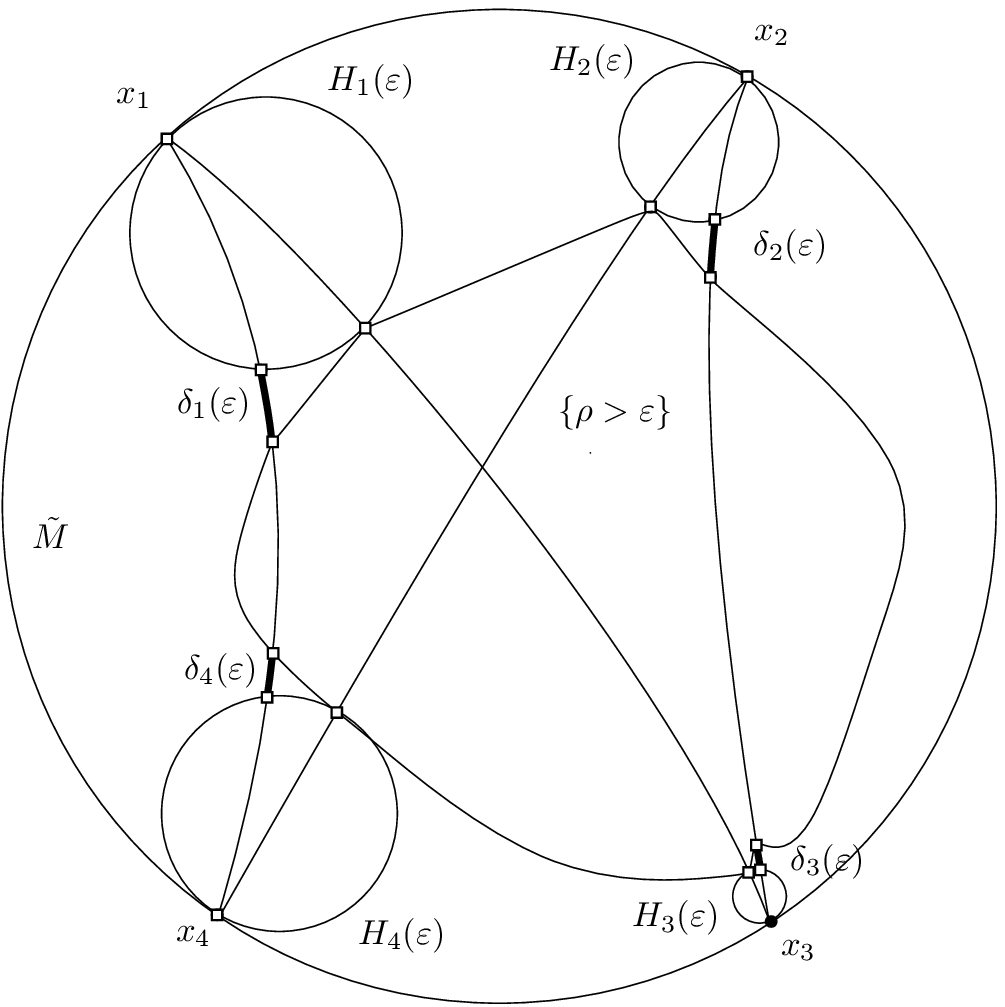}  \hfill
\includegraphics[scale=0.8]{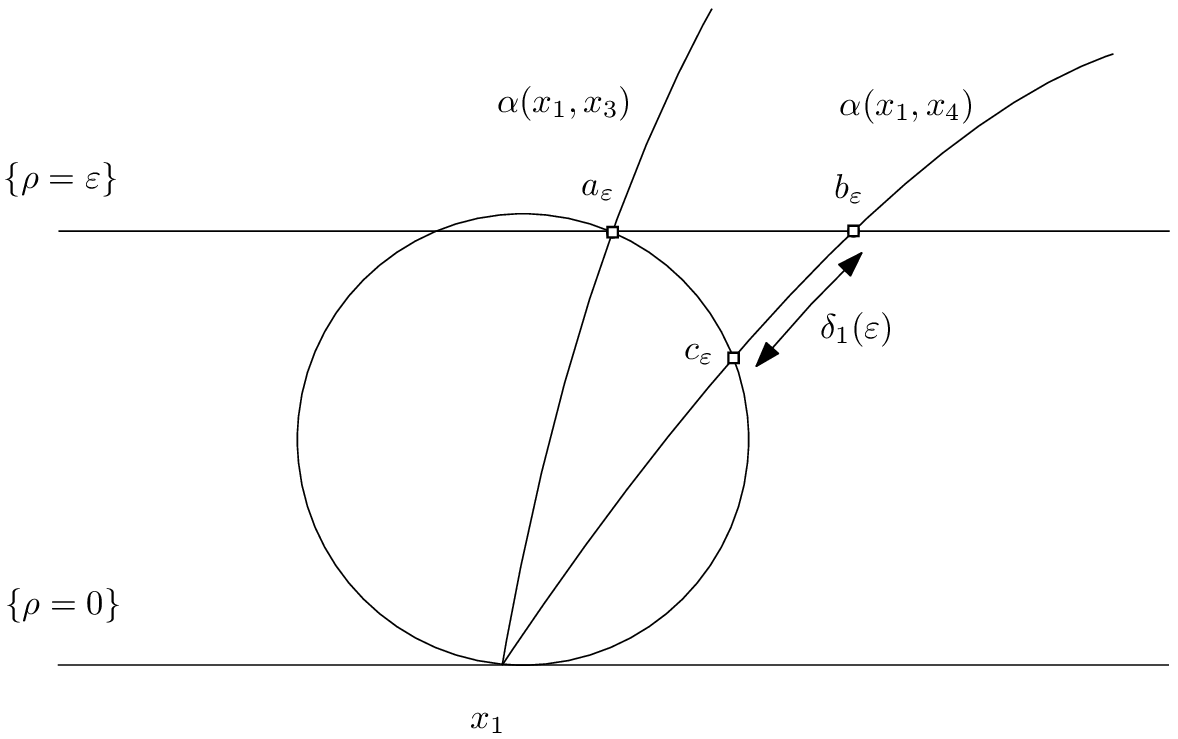} 
\caption{Left: The four horospheres and the lengths $\delta_i(\eps)$. Right: The horosphere $H_1(\eps)$}

\label{fig:horo}

\end{center}
\end{figure}

We have:
\[ \begin{split} &  L(x_1,x_3) + L(x_2,x_4) - L(x_2,x_3) - L(x_1,x_4) \\
& = \lim_{\eps \rightarrow 0} l(\alpha(x_1,x_3) \cap \left\{ \rho > \eps \right\}) + 2\ln \eps + l(\alpha(x_2,x_4) \cap \left\{ \rho > \eps \right\}) + 2\ln \eps \\  & \hspace{3cm}  - l(\alpha(x_2,x_3) \cap \left\{\rho > \eps \right\}) - 2\ln \eps   - l(\alpha(x_1,x_4) \cap \left\{\rho > \eps \right\}) -  2\ln \eps \\ & = \lim_{\eps \rightarrow 0} l(\alpha(x_1,x_3) \cap \left\{\rho > \eps \right\}) + l(\alpha(x_2,x_4) \cap \left\{\rho > \eps \right\})- l(\alpha(x_2,x_3) \cap \left\{\rho > \eps \right\})   \\ & \hspace{3cm} - l(\alpha(x_1,x_4) \cap \left\{\rho > \eps \right\}) \\
& = \lim_{\eps \rightarrow 0} l(\alpha(x_1,x_3) \cap H_{\text{ext}}(\eps)) + l(\alpha(x_2,x_4) \cap H_{\text{ext}}(\eps)) - l(\alpha(x_2,x_3) \cap H_{\text{ext}}(\eps)) \\ &  \hspace{3cm}  - l(\alpha(x_1,x_4) \cap H_{\text{ext}}(\eps)) - \delta_1(\eps) - \delta_2(\eps) - \delta_3(\eps) - \delta_4(\eps), \end{split} \]
where $\delta_i(\eps)$ is the algebraic distance on the geodesic between its intersection with $H_i(\eps)$ and $\left\{\rho=\eps\right\}$, positively counted from $x_i$, and $H_{\text{ext}}(\eps) := \wt{M} \setminus \cup_{i=1}^4 H_i(\eps)$. Now, we know that the quantity
\[ \begin{split}  l(\alpha(x_1,x_3) \cap H_{\text{ext}}(\eps)) + l(\alpha(x_2,x_4) \cap H_{\text{ext}}(\eps)) - l(\alpha(x_2,x_3) \cap H_{\text{ext}}(\eps)) - l(\alpha(x_1,x_4) \cap H_{\text{ext}}(\eps)) \end{split} \]
is actually independent of $\eps$ (it is easy to check) and amounts to $\eta(\Omega)$ (see \cite{aw} for instance). It is thus sufficient to prove that $\delta_i(\eps) \rightarrow 0$ as $\eps \rightarrow 0$. Let us consider $\delta_1(\eps)$ for instance and $\eps$ small enough so that we can work in the coordinates where the metric $g$ can be written in the form
\[ g = \dfrac{d\rho^2 + h^2(\rho,y) dy^2}{\rho^2}, \]
for some smooth positive function $h^2$ (down to the boundary).





We have:
\[ \delta_1(\eps) = d(c_\eps, b_\eps) \leq d(c_\eps, a_\eps) + d(a_\eps, b_\eps) \leq d(c_\eps, a_\eps) + l([a_\eps, b_\eps]), \]
where the points $a_\eps, b_\eps, c_\eps, d_\eps$ are introduced in Figure \ref{fig:horo}, $[a_\eps, b_\eps]$ denotes the euclidean segment joining $a_\eps$ to $b_\eps$. Note that by construction $d(c_\eps, a_\eps) \rightarrow 0$ as $\eps \rightarrow 0$ (the points are on the same family of shrinking horospheres).

The two geodesics $\alpha(x_1,x_3)$ and $\alpha(x_1,x_4)$ with endpoint $x_1$, seen as curves in $\widetilde{M}$, can be locally parametrized by the respective smooth functions $(\rho, y_3(\rho))$ and $(\rho,y_4(\rho))$, according to the implicit function theorem since the geodesics intersect transversally the boundary (see Lemma \ref{lem:xbar}). One has by derivating at $\rho = 0$ that $\lambda_i \partial_{\rho} = \partial_{\rho} + y_i'(0) \partial_y$ for some constant $\lambda_i$, that is $y_i'(0) = 0$ and $\lambda_i=1$. In other words, we can parametrize locally both geodesics by $(\rho, y_0 + \mathcal{O}(\rho^2))$, where $y_0$ is some constant depending on the choice of coordinates. Thus $|y(a_\eps)-y(b_\eps)| = \mathcal{O}(\eps^2)$. If we choose a parametrization $\gamma(t) = (\eps, y(a_\eps) + t(y(b_\eps) - y(a_\eps)))$, for $t \in [0,1]$, of the euclidean segment $[a_\eps, b_\eps]$, then one has:
\[ l([a_\eps, b_\eps]) = \int_0^1 g(\dot{\gamma}(t), \dot{\gamma}(t))^{1/2} dt = \dfrac{|y(b_\eps) - y(a_\eps)|}{\eps} \int_0^1 h(\gamma(t)) dt, \]
where the integral is uniformly bounded with respect to $\eps$. Thus, by the previous remarks, $l([a_\eps, b_\eps]) = \mathcal{O}(\eps)$, which concludes the proof.

\end{proof}

\section{Construction of the deviation $\kappa$}

\label{sect:dev}

In this section, for the sake of simplicity, we will sometimes write $A = \mathcal{O}(\eps^\infty)$ in order to denote the fact that for all $n \in \N^*$, there exists $C_n > 0, \eps_n > 0$ such that: $\forall \eps \leq \eps_n, |A| \leq C_n \eps^n$.

\subsection{Reducing the problem}

Suppose $g_1$ and $g_2$ are two asymptotically hyperbolic metrics like in the setting of Theorem \ref{th}, that is they are either both negatively-curved and their renormalized distance coincide for some choices of conformal representatives.

According to \cite[Theorem 2]{ggsu} we know that there exists a smooth diffeomorphism $\psi : \overline{M} \rightarrow \overline{M}$ which is the identity on $\partial M$ and such that $\hat{g}_1 := \psi^* g_1 = g_2 + \mathcal{O}(\rho^\infty)$ (in the sense that the conformal representative of the two metrics coincide to infinite order on the boundary), where $\rho$ is the boundary defining function induced by the choice of $h_2$.

Notice that $(M,\hat{g}_1)$ satisfies the same assumption as $(M,g_1)$ and that their renormalized lengths agree since $\psi$ restricts to the identity on the boundary, i.e. $L_{g_1} = L_{\hat{g}_1}$, where both lengths are computed with respect to the same representative for the conformal infinity (this follows from Remark \ref{rem:psi}). In the following, for the sake of simplicity, we will actually denote $\hat{g}_1$ by $g_1$ and argue on this new metric. Thus, in these new notations, we have $g_1 = g_2 + \mathcal{O}(\rho^\infty)$.

\begin{rema} \label{rem:infty} In particular, this implies that the respective renormalized vector fields satisfy $\overline{X}_1 = \overline{X}_2 + \mathcal{O}(\rho^\infty)$, that is their $\mathcal{C}^\infty$-jet coincide on the boundary. By Duhamel's formula (see \cite[Lemma 2.2]{suv2} for instance) this implies that on the boundary $\partial_- S^*M$, for any $k \geq 0$, one has $||\overline{\varphi}^1_\tau - \overline{\varphi}^2_\tau||_{\mathcal{C}^k} = \mathcal{O}(\tau^\infty)$. \end{rema}

\subsection{The diffeomorphism $\kappa$}

\label{ssect:kappa}

We denote by $M_\eps := M \cap \left\{ \rho \geq \eps\right\}$ and by $\widetilde{M}_\eps$ its lift to the universal cover. Like in the previous section, all the objects are lifted on the universal cover. Unless it is mentioned, we will drop the notation $\wt{\cdot}$, except for the universal cover itself. $S^*\wt{M}_i$ will denote the unit cotangent bundle with respect to the metric $g_i$. $\mathcal{G}_1$ and $\mathcal{G}_2$ denote the set of geodesics connecting points on the ideal boundary $\partial \wt{M}$, with respect to the metrics $g_1$ and $g_2$. They will sometimes be identified with $\partial \wt{M} \times \partial \wt{M} \setminus \text{diag}$.

Given $(x,\xi) \in S^*\widetilde{M}_1 \setminus \Gamma^1_-\cup \Gamma^1_+$, we denote by $(z,z') \in \partial \wt{M} \times \partial \widetilde{M}$ (resp. $(y,y') \in \partial \wt{M} \times \partial \widetilde{M}$) the two points on the ideal boundary induced by the geodesic carrying the point $(x,\xi)$ (resp. $(x,R_\theta \xi)$ if $\theta \in (0,\pi)$ and $(x,R_\theta \xi) \in S^*\widetilde{M}_1 \setminus \Gamma^1_-\cup \Gamma^1_+$). This defines a map:
\[ \kappa_1 : \left| \begin{array}{l} \wt{W}_1 \rightarrow \mathcal{G}_1 \times \mathcal{G}_1 \setminus \text{diag} \\ (x,\xi,\theta) \mapsto (z,z',y,y') \end{array} \right., \]
where
\[ \wt{W}_1 := \left\{(x,\xi,\theta) \in S^*\wt{M}_1 \times (0,\pi), (x,\xi), (x,R_\theta \xi) \notin (\Gamma^1_- \cup \Gamma^1_+)\right\} \]
The map $\kappa_1$ is clearly bijective. It is smooth because each of the coordinates $(z,z',y,y')$ is smooth. Indeed, one has for instance
\[ z(x,\xi,\theta) = \pi_0(\overline{\varphi}^1_{\tau_-(x,\xi)}(x,\xi)), \]
and this is a smooth application according to Section \ref{sssect:trapped}.

The $g_2$-geodesics with endpoints $(z,z')$ and $(y,y')$ intersect in a single point $(\wt{\mathrm{x}}(x,\xi,\theta), \wt{\Xi}(x,\xi,\theta))$ (where $\wt{\Xi}$ is the covector on the $g_2$-geodesic with endpoints $(z,z')$) and form an angle $\wt{f}(x,\xi,\theta)$, which we call the \textit{angle of deviation}. This defines a map
\be \label{eq:kappa} \wt{\kappa} := \kappa_2^{-1} \circ \kappa_1 : \left| \begin{array}{l} \wt{W}_1 \rightarrow \wt{W}_2 \\
(x,\xi,\theta) \mapsto (\wt{\mathrm{x}}(x,\xi,\theta), \wt{\Xi}(x,\xi,\theta), \wt{f}(x,\xi,\theta)) \end{array}\right. \ee
where $\wt{W}_2$ is defined in the same fashion as $\wt{W}_1$. By the implicit function theorem, one can prove that $\kappa^{-1}_2$ is smooth and thus $\wt{\kappa}$ too. It is a bijective map whose inverse $\wt{\kappa}^{-1} = \kappa^{-1}_1 \circ \kappa_2$ is smooth by the same arguments. As a consequence, $\wt{\kappa}$ is a smooth diffeomorphism. Moreover, it is invariant by the action of the fundamental group and thus descends to the base as an application $\kappa : (x,\xi,\theta) \mapsto (\mathrm{x},\Xi,f)$.

\subsection{Scattering on the universal cover}

On the universal cover $\wt{M}$, the renormalized distance can actually be extended outside the boundary, namely we can set:
\[ D(p,q) := d(p,q)+\ln(\rho(p))+\ln(\rho(q))\]
$D$ is clearly smooth on $\wt{M} \times \wt{M} \setminus \text{diag}$ and using the fact that there exists a unique geodesic connecting two points, one can prove like in \cite[Proposition 5.15]{ggsu}, that the extension of $D$ to $\wt{\overline{M}} \times \wt{\overline{M}} \setminus \text{diag}$ is smooth. Now, as established in \cite[Proposition 5.16]{ggsu} the renormalized distance on the boundary actually determines the scattering map $\sigma$ (defined in (\ref{eq:scat})), that is:

\begin{prop}
\label{prop:scat}
If $L_1 = L_2$, then $\sigma_1 = \sigma_2$.
\end{prop}

The proof also applies here, in the universal cover. It is a standard computation since we know that $D$ is differentiable, which relies on the fact that the gradient of $q \mapsto L_i(p,q)$ (for $p, q \in \partial \wt{M}$) is the projection on the tangent space $T_q \partial \wt{M}$ of the gradient of $q \mapsto D_i(p,q)$, which precisely corresponds to the direction of the geodesic joining $p$ to $q$ when it exits $\wt{M}$.

\label{ssect:psi}

We fix $\eps > 0$ and define $S^*\wt{M}^i_\eps := S^*\wt{M}_i \cap \left\{\rho\geq\eps\right\}$. For $i \in \left\{1,2\right\}$, given $(x,\xi) \in \partial_-S^*\widetilde{M}^i_{\eps}$ we can represent the vector $\xi=\xi(\omega)$ by the angle $\omega \in [0,\pi]$ such that $\sin \omega = |g_i(\nu_i(x),\xi)|$, where $\nu_i$ stands for the unit outward normal covector to $\left\{\rho=\eps\right\}$ (with respect to the metric $g_i$).

\begin{lemm}
\label{lem:est1}
There exists an angle $\omega_\eps$ (only depending on $\eps$), such that for all $(x, \xi(\omega)) \in \partial_-S^*\widetilde{M}^1_{\eps} \setminus \Gamma_-^1$, given by an angle $\omega \in [\omega_\eps, \pi-\omega_\eps]$, if $\alpha_1(p,q)$ denotes the $g_1$-geodesic generated by $(x,\xi)$, with endpoints $(p, q) \in \partial \widetilde{M} \times \partial \wt{M}$, then the $g_2$-geodesic $\alpha_2(p,q)$ with endpoints $p$ and $q$ intercepts the set $\left\{\rho>\eps\right\}$. Moreover, for any $N \in \N^*$, we can take $\omega_\eps =\eps^N$.
\end{lemm}

\begin{proof}[Proof]
Let $(x,\xi) \in \partial_-S^*\wt{M}_\eps^1$. We set ourselves in the coordinates $(\rho,y)$ induced by the conformal representative $h$. The trajectory
\[ t \mapsto (\rho(t),y(t),{\xi_0}(t),\eta(t)) \in S^*\wt{M} \]
of the point $(x,\xi)$ under the flow $X$ is given by Hamilton's equation (see \cite[Equation (2.8)]{ggsu}). Flowing backwards in time with $\varphi_t$, we know that $(x,\xi)$ converges exponentially fast towards a point $(p,\zeta) \in \partial_-S^*\wt{M}$ (see \cite[Equation (2.11)]{ggsu}) in the sense that there exists a constant $C$ (uniform in the choice of points) such that:
\[ \forall t \leq 0, ~~\rho(t)\leq C\rho(0)e^{-|t|} = \eps C e^{-|t|} \]
In particular, the time $\tau_-(x,\xi)$ taken by the point $(x,\xi)$ to reach $(p,\zeta)$ with the flow $\overline{\varphi}_\tau^1$ is (see equation (\ref{eq:xbar})):
\[ \tau_-(x,\xi) = \int_{-\infty}^0 \rho(t) dt \leq C \eps \]

We also know, according to Hamilton's equations (see (2.8) in \cite{ggsu}) that
\[\dot{\rho}(0) = \rho^2(0) \xi_0(0) = \eps \sin(\omega),\]
where $\omega$ satisfies $\overline{\xi}_0(0) = \rho \xi_0(0) = \sin(\omega) = |g_1(\xi,\nu_1(x))|$. Let us fix an integer $N > 0$ and assume that $\eps^N \leq \omega \leq \pi-\eps^N$. Then $\dot{\rho}(0) \geq 2/\pi \cdot \eps^{N+1}$ so there exists an interval $[0,\delta]$ such that for $t \in [0,\delta]$:
\[  \eps + t/\pi \cdot \eps^{N+1} \leq \eps + t/2 \cdot \dot{\rho}(0) \leq \rho(t) \leq 2\eps\]
In particular, $\rho(\delta) \geq \eps + \delta/\pi \cdot \eps^{N+1}$.

We go back to the flow $\overline{\varphi}_\tau^1$. By our previous remark, we know that there exists a time
\[ \tau_0 \leq C\eps+\int_0^\delta \rho(t)dt \leq C' \eps, \]
such that $\rho(\overline{\varphi}_{\tau_0}^1(p,\zeta)) \geq \eps + \delta/\pi \cdot \eps^{N+2}$. But since $g_1 = g_2 + \mathcal{O}(\rho^\infty)$, we know that $X_1 = X_2 + \mathcal{O}(\rho^\infty)$ and $\overline{X}_1 = \overline{X}_2 + \mathcal{O}(\rho^\infty)$. Moreover, since the scattering maps agree according to Proposition \ref{prop:scat}, we know that the two geodesics $\alpha_1(p,q)$ and $\alpha_2(p,q)$ are both generated by $(p,\zeta)$. As a consequence, one has: $\rho(\overline{\varphi}_\tau^1(p,\zeta)) = \rho(\overline{\varphi}_\tau^2(p,\zeta))+\mathcal{O}(\tau^\infty)$ (the remainder being independent of $(p,\zeta)$). In particular, since $\tau_0 \leq C'\eps$, there exists a constant $C'' > 0$ such that
\[ |\rho(\overline{\varphi}_{\tau_0}^1(p,\zeta)) - \rho(\overline{\varphi}_{\tau_0}^2(p,\zeta))|\leq C''\eps^{N+2}\]
Thus:
\[ \rho(\overline{\varphi}_{\tau_0}^2(p,\zeta)) \geq \eps + \frac{\delta}{\pi}\eps^{N+1}-C''\eps^{N+2} > \eps, \]
if $\eps$ is small enough.  

\end{proof}

\begin{wrapfigure}[14]{r}{5cm}
\includegraphics[width=5cm]{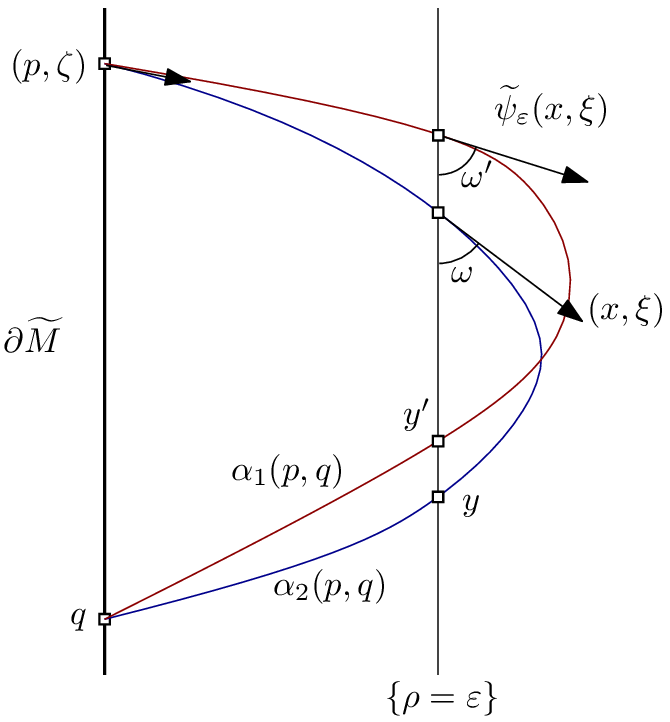}
\caption{The diffeomorphism $\wt{\psi}_\eps$}
\end{wrapfigure}
In the following, we assume that such an integer $N$ is fixed (and taken large enough) and we apply the previous lemma with $N+1$, that is $\omega_\eps = \eps^{N+1}$. \\

This allows us to define a map $\wt{\psi}$ on $\mathcal{U} := \left\{ (x,\xi(\omega)) \in S^*\widetilde{M}^1,\overline{\xi}_0 \geq 0, \omega \in [\rho(x)^{N+1}, \pi-\rho(x)^{N+1}] \right\}$, in the following way: to a point $(x,\xi) \in \mathcal{U}$, which we see as a boundary point $(x,\xi(\omega)) \in \partial_-S^*\widetilde{M}^1_{\eps}$ for $\eps = \rho(x)$, we associate the boundary point $(x',\xi') = \wt{\psi}(x,\xi)$ such that $\wt{\psi}(x,\xi) \in \partial_-S^*\widetilde{M}^2_{\eps}$ is the point on the $g_2$-geodesic connecting $p$ to $q$. A formal way to define $\wt{\psi}$ is to introduce another diffeomorphism $\mathrm{p}_1 : \mathcal{U} \rightarrow \partial_-S^*\wt{M} \times [0,\infty)$ such that $\mathrm{p}_1(x,\xi) = \left(\overline{\varphi}^1_{\tau_-(x,\xi)}(x,\xi),\rho(x)\right)$ and to set
\be \wt{\psi}(x,\xi) = \mathrm{p}^{-1}_2 \circ \mathrm{p}_1 (x,\xi) = \overline{\varphi}^2_{\tau_\rho}\left(\overline{\varphi}^1_{\tau_-(x,\xi)}(x,\xi)\right),\label{eq:p1p2}\ee
where $\mathrm{p}_2$ is defined in the same fashion and $\tau_\rho$ is the time taken to reach the hypersurface $\left\{\rho=\rho(x)\right\}$. Note that $\wt{\psi}(x,\xi)$ exists according to the previous lemma and this point is well-defined (it is unique) according to Lemma \ref{lem:conv}. Moreover, it is smooth on $\mathcal{U}$ thanks to the results of Section \ref{sssect:trapped} (this mainly follows from the implicit function theorem). Eventually, it is invariant by the action of the fundamental group and descends on the base as a map $\psi$. We write $\mathcal{U}_\eps := \mathcal{U} \cap \left\{\rho=\eps\right\}$. What we need, is to prove that $\wt{\psi}$ is the identity plus a small remainder or, more precisely, if we denote by $\wt{\psi}_\eps = \wt{\psi}|_{\mathcal{U}_\eps}$, that $\wt{\psi}_\eps = \text{Id} + \mathcal{O}(\eps^\infty)$ in the $\mathcal{C}^1$-topology.




\begin{lemm}\label{lem:est2}
$||\wt{\psi}_\eps-\text{Id}~ ||_{\mathcal{C}^1} = \mathcal{O}(\eps^\infty)$.
\end{lemm}

\begin{proof}
Since the two trajectories are $\mathcal{O}(\eps^\infty)$ close, so will be the times $\tau_\rho$ and $-\tau_-(x,\xi)$ by which the $g_1$- and $g_2$-geodesics generated by $(p,\zeta)$ hit $\left\{\rho=\eps\right\}$ (this can be proved by contradiction for instance, like in the proof of Lemma \ref{lem:est1}), which implies that $\wt{\psi}_\eps(x,\xi)=(x,\xi) + \mathcal{O}(\eps^\infty)$, where the remainder is uniform in $(x,\xi)$. To obtain a bound on the derivatives, we see from the expression (\ref{eq:p1p2}) and the fact that the two flows are $\mathcal{O}(\eps^\infty)$ close in the $\mathcal{C}^1$-topology (Remark \ref{rem:infty}), that it is sufficient to show that the times satisfy $\tau_\rho(x,\xi) = -\tau_-(x,\xi) + \mathcal{O}(\eps^\infty)$ in the $\mathcal{C}^1$-topology with a uniform remainder. Let $(p,\zeta) = \overline{\varphi}^1_{\tau_-(x,\xi)}(x,\xi)$. We have
\[ \rho(\overline{\varphi}^1_{-\tau_-(x,\xi)}(p,\zeta)) = \eps = \rho(\overline{\varphi}^2_{\tau_\rho}(p,\zeta)) \]
We are interested in the variations of $x$ along $\left\{\rho=\eps\right\}$ and of the angle $\xi(\omega)$. If we denote by $z$ any of these two parameters, we get by derivating the previous equality:
\[ -\dfrac{\partial \tau_-}{\partial z} d\rho(\overline{X}_1) + d\rho(d\overline{\varphi}^1_{-\tau_-}(d_z(p,\zeta))) = \dfrac{\partial \tau_\rho}{\partial z} d\rho(\overline{X}_2) + d\rho(d\overline{\varphi}^2_{\tau_\rho}(d_z(p,\zeta)))\]
The two terms containing the differential of the flow coincide to order $\mathcal{O}(\eps^\infty)$ and we also have $d\rho(\overline{X}_2) = d\rho(\overline{X}_1) + \mathcal{O}(\eps^\infty)$ by Remark \ref{rem:infty}. Thus:
\[ \left( -\dfrac{\partial \tau_-}{\partial z} - \dfrac{\partial \tau_\rho}{\partial z}\right)d\rho(\overline{X}_1) = \mathcal{O}(\eps^\infty) \]
But $d\rho(\overline{X}_1)$ is precisely the sine of the angle with which the geodesic generated by $(p,\zeta)$ enters the set $\left\{\rho\geq\eps\right\}$ and this angle is contained in $[\eps^N,\pi-\eps^N]$ by construction of the set $\mathcal{U}$, so $d\rho(\overline{X}_1) \geq \eps^N$. By dividing by $d\rho(\overline{X}_1)$, this term is swallowed in the $\mathcal{O}(\eps^\infty)$, which provides the sought result.
\end{proof}

Given $(x, \xi) \in \partial_-S^*\wt{M}^i_\eps$, we denote by $l^i_{\eps,+}(x,\xi)$ the length of the geodesic generated by this point in $\wt{M}_\eps$. Note that by strict convexity of the sets $\left\{\rho\geq\eps\right\}$ the intersections of the geodesics (for both metrics) with $\wt{M}_\eps$ have a single connected component, so this length is well-defined.

\begin{lemm}
\label{lem:est3}
$||l^1_{\eps,+} - l^2_{\eps,+}\circ\wt{\psi}_\eps||_{\mathcal{C}^0} =\mathcal{O}(\eps^\infty)$, where the sup is computed over $\partial_-S^*\widetilde{M}^1_{\eps} \setminus \Gamma_-^1$.
\end{lemm}

\begin{proof}[Proof]
Recall that $(p,\zeta) \in \partial_-S^*\wt{M}$ is the point obtained by flowing backwards $(x,\xi)$ down to the boundary. If $D_i$ denotes the renormalized distance for both metrics, then we have:
\[ D_1(p,x) = D_2(p,x'(x,\omega)) + \mathcal{O}(\eps^\infty), \]
where the remainder is independent of $(x,\xi)$. Indeed, considering $0 < \eps' < \eps$, and denoting by $\alpha_1(p,x)$ the $g_1$-geodesic joining $p$ to $x$, one has:
\[ \begin{split} l_1(\alpha_1(p,x) \cap \left\{\rho > \eps'\right\}) + \ln \eps' & = \int_{\tau^1_{\eps'}}^{\tau^1_\eps} \dfrac{ds}{\rho(\overline{\varphi}^1_s(z))} +  \ln \eps' \\ & = \int_{\eps'}^\eps \dfrac{(\psi_1^{-1})'(u)du}{u} + \ln \eps', \end{split} \]
where $\tau^1_\eps$ and $\tau^1_{\eps'}$ are defined such that $\rho(\overline{\varphi}^1_{\tau^1_\eps}(z))=\eps, \rho(\overline{\varphi}^1_{\tau^1_{\eps'}}(z))=\eps'$, and $\psi_1 : s \mapsto \rho(\overline{\varphi}^1_s(z))$ is a diffeomorphism. Note that $\psi_1(0)=0, \psi_1'(0)=1$. By assumption, the two metrics are close, thus $\psi_1(s) = \psi_2(s) + \mathcal{O}(s^\infty)$ and one can check (by induction) that this implies that $(\psi_1^{-1})^{(k)}(0)=(\psi_2^{-1})^{(k)}(0)$ for all $k \in \N$, that is $\psi_1^{-1}(u) = \psi_2^{-1}(u) + \mathcal{O}(u^{\infty})$. Inserting this into the previous integral expression, we get the claimed result.

The same occurs for the other bits of the geodesics: namely, if $y$ and $y'$ denote the exit points of $\alpha_1(p,q)$ and $\alpha_2(p,q)$ in $\wt{M}_\eps$, then $D_1(q,y) = D_2(q,y') + \mathcal{O}(\eps^\infty)$. Now, using the fact that the renormalized lengths agree on the boundary, we obtain:
\[ \begin{split} D_1(p,q) & = D_1(p,x) + d_1(x,y) + D_1(y,q) \\
& = D_1(p,x) + l_{\eps,+}^1(x,\xi) + D_1(y,q) \\
& = D_2(p,q) \\
& = D_2(p,x') + l_{\eps,+}^2(\wt{\psi}_\eps(x,\xi)) + D_2(y',q) \end{split} \]
Thus: $l^1_{\eps,+}(x,\xi) = l^2_{\eps,+}(\wt{\psi}_\eps(x,\xi)) + \mathcal{O}(\eps^\infty)$.
\end{proof}

\subsection{The average angle deviation}

\label{ssect:av}

The angle of deviation $\wt{f}$ satisfies two elementary properties:

\begin{lemm}
\label{lem:f}
\begin{enumerate}
\item It is $\pi$-symmetric, that is, for almost all $(x,\xi) \in S^*\widetilde{M}_1, \theta \in [0,\pi]$,
\be \label{eq:pisym} \wt{f}(x,\xi,\theta) = \pi - \wt{f}(x,R_\theta \xi,\pi-\theta) \ee
\item It is superadditive in the sense that, for almost all $(x,\xi) \in S^*\widetilde{M}_1, \theta_1, \theta_2 \in [0,\pi]$ such that $\theta_1 + \theta_2 \in [0,\pi]$,
\be \label{eq:superadd} \wt{f}(x,\xi,\theta_1) + \wt{f}(x,R_{\theta_1}\xi,\theta_2) \leq \wt{f}(x,\xi,\theta_1+\theta_2) \ee
\end{enumerate}
\end{lemm}

We will denote by $\mathrm{h} : \mathcal{G}_1 \rightarrow \mathcal{G}_2$ the map that associates to a $g_1$-geodesic with endpoints $z, z' \in \widetilde{M}$ the $g_2$-geodesic with same endpoints. Note that when $\mathcal{G}_1$ and $\mathcal{G}_2$ are identified with $\partial \wt{M} \times \partial \wt{M}$, $\mathrm{h}$ is simply the identity, but we will rather see $\mathcal{G}_i$ as the set of geodesics connecting two boundary points. 

\begin{proof}[Proof]
The $\pi$-symmetry is obtained from the very definition of $\wt{f}$. As to the superadditivity, it follows from Gauss-Bonnet formula in negative curvature. Indeed, consider the three geodesics $\alpha_1,\beta_1,\gamma_1$ of $\widetilde{M}_1$, respectively carried by the points $(x,\xi), (x,R_{\theta_1}\xi), (x,R_{\theta_1+\theta_2}\xi)$. Their image by $\mathrm{h}$ (that is the corresponding $g_2$-geodesics with same endpoints) are three geodesics $\alpha_2 = h(\alpha_1), \beta_2 = h(\beta_2), \gamma_2 = h(\gamma_2)$, forming a geodesic triangle which we denote by $T$, with angles 
\[\wt{f}(x,\xi,\theta_1), \wt{f}(x,R_{\theta_1}\xi,\theta_2), \wt{f}(x,R_{\theta_1+\theta_2}\xi,\pi-\theta_1-\theta_2)\]
Now, we have by Gauss-Bonnet formula:
\be 0 \geq \int_T \kappa ~~ d\text{vol}_g = \wt{f}(x,\xi,\theta_1) + \wt{f}(x,R_{\theta_1}\xi,\theta_2) + \wt{f}(x,R_{\theta_1+\theta_2}\xi,\pi-\theta_1-\theta_2)  - \pi \ee
Using $\pi$-symmetry, we obtain inequality (\ref{eq:pisym}). 
\end{proof}

Note that the inequality (\ref{eq:superadd}) is saturated if and only if the geodesic triangle is degenerate, that is it is reduced to a single point, since the curvature is negative. As mentioned previously, $\wt{f}$ descends on the base as a function $f$ which also satisfies the properties of Lemma \ref{lem:f}.

One of the ideas of Otal was to introduce the \textit{average angle of deviation}. Since we work in a non-compact setting, we are forced to consider partial averages depending on $\eps$. We define for fixed $\eps > 0$:
\be \Theta_\eps(\theta) := \dfrac{1}{\text{vol}_{g_1}(S^*M^1_\eps)} \int_{S^*M^1_\eps} f(x,\xi,\theta) d\mu_1(x,\xi) \ee
It also satisfies
\be \label{eq:bord} \Theta_\eps(0) = 0, \Theta_\eps(\pi)=\pi \ee
Since the rotations $R_\theta$ preserve the Liouville measure, by integrating over $S^*M^1_\eps$ the relations (\ref{eq:pisym}) and (\ref{eq:superadd}) given in Lemma \ref{lem:f}, we see that $\Theta_\eps$ also satisfies the $\pi$-symmetry: 
\be \label{eq:pisym2} \forall \theta \in [0,\pi], ~~~ \Theta_\eps(\theta) = \pi - \Theta_\eps(\pi-\theta), \ee
and the superadditivity:
\be \label{eq:superadd2} \forall \theta_1, \theta_2 \in [0,\pi], \text{ s.t. } \theta_1 + \theta_2 \in [0,\pi], ~~~ \Theta_\eps(\theta_1) + \Theta_\eps(\theta_2) \leq \Theta_\eps(\theta_1 + \theta_2) \ee

We now show that $\Theta_\eps$ satisfies the following

\begin{lemm}
\label{lem:est4}
Let $J : [0, \pi] \rightarrow \R$ be a convex continuous function. Then:
\be \label{eq:conv} \int_0^\pi J(\Theta_\eps(\theta)) \sin(\theta) d\theta \leq \int_0^\pi J(\theta) \sin(\theta) d\theta + ||J||_{L^\infty} \mathcal{O}(\eps^N), \ee
where the remainder only depends on $\eps$, $N$ is fixed by Lemma \ref{lem:est1}.
\end{lemm}

\begin{wrapfigure}[15]{r}{8cm}
\includegraphics[width=7cm]{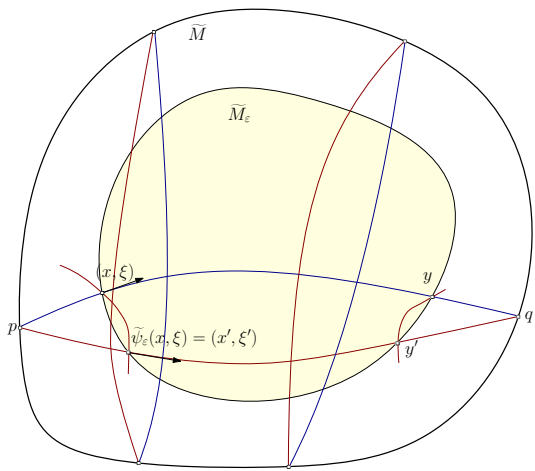}
\caption{A picture of the situation : in red, the $g_2$-geodesics, in blue, the $g_1$-geodesics}
\label{fig:defect}
\end{wrapfigure}
The proof of this lemma relies on the use of Santalo's formula, together with the fact that the Liouville currents coincide. But let us make a preliminary remark. Consider $(x, \xi(\omega)) \in \partial_-S^*\widetilde{M}^1_{\eps}$ with $\omega \in [\omega_\eps,\pi-\omega_\eps]$. It generates the $g_1$-geodesic $\alpha_1(p,q)$ with endpoints $p, q \in \partial \widetilde{M}$ which enters (resp. exits) $\widetilde{M}_\eps$ at $x$ (resp. $y$). We denote by $\alpha_2$ the $g_2$-geodesic joining $p$ and $q$ which enters (resp. exits) $\widetilde{M}_\eps$ at $x'=x'(\wt{\psi}_\eps(x,\xi))$ (resp. $y'$). Let us denote by $\mathcal{F}_1(x,y) \subset \mathcal{G}$ the $g_1$-geodesics which have a positive transverse intersection with the geodesic segment $\alpha^1_\eps := \alpha_1 \cap \widetilde{M}_\eps$. $\mathcal{F}_2(x',y')$ denotes its analogue for the second metric, that is the $g_2$-geodesics having a positive transverse intersection with $\alpha^2_\eps := \alpha_2 \cap \widetilde{M}_\eps$.

Since $\mathrm{h}$ preserves the Liouville measure (that is $\mathrm{h}_*\eta_1=\eta_2$), we have:
\[ \eta_1(\mathcal{F}_1(x,y)) = \eta_2(\mathrm{h}(\mathcal{F}_1(x,y))) \]
We could hope that $\mathrm{h}(\mathcal{F}_1(x,y))=\mathcal{F}_2(x',y')$ but this is not the case (see Figure (\ref{fig:defect})), insofar as there is a slight defect due to the fact that we are not looking at points on the boundary, and this is where the arguments of Otal fail to apply immediately. However, we have:

\begin{lemm}
\label{lem:est5}
\[\eta_1(\mathcal{F}_1(x,y)) = \eta_2(\mathcal{F}_2(x',y')) + \mathcal{O}(\eps^\infty),\]
where the remainder is independent of $(x,\xi)$.
\end{lemm}

\begin{proof}
It follows from Lemma \ref{lem:est3}, combined with equation (\ref{eq:d}).
\end{proof}

We can now establish the lemma on convexity. We will be careful to use the notation $\wt{\cdot}$ to refer to the objects on the universal cover.

\begin{proof}[Proof]
$d\mu_1/\text{vol}_{g_1}(S^*M^1_\eps)$ is a probability measure on $S^*M^1_\eps$ and by Jensen inequality, we have, for all $\theta \in [0,\pi]$:
\[ J(\Theta_\eps(\theta)) \leq \dfrac{1}{\text{vol}_{g_1}(S^*M^1_\eps)} \int_{S^*M^1_\eps} J(f(x,\xi,\theta)) d\mu_1(x,\xi) \]
Multiplying by $\sin(\theta)$, integrating over $[0,\pi]$ and applying Fubini's Theorem, we obtain:
\[ \int_0^\pi J(\Theta_\eps(\theta)) \sin(\theta) d\theta \leq \dfrac{1}{\text{vol}_{g_1}(S^*M^1_\eps)} \int_{S^*M^1_\eps} \int_0^\pi J(f(x,\xi,\theta)) \sin(\theta) d\theta d\mu_1(x,\xi)   \]
Using Santalò's formula, we obtain for the last integral:
\[ \begin{split}
\int_{S^*M^1_\eps} \int_0^\pi & J(f(x,\xi,\theta)) \sin(\theta) d\theta d\mu_1(x,\xi)  \\ &=  \int_{\partial_- S^*M^1_\eps} \int_{0}^{l_{\eps,+}^1(x,\xi)} \int_0^\pi J(f(\varphi_\tau^1(x,\xi),\theta)) \sin(\theta) d\theta d\tau d\mu_{1,\nu}(x,\xi) \end{split}, \]
where $d\mu_{1,\nu}(x,\xi) = |g_1(\xi,\nu_1)|i^*_{\partial S^*M^1_\eps}(d\mu_1)$, $\nu_1$ is the normal unit outward covector to the boundary, $i^*_{\partial S^*M^1_\eps}(d\mu_1)$ is the restriction of the Liouville measure to the boundary (the measure induced by the Sasaki metric restricted to $\partial S^*M^1_\eps$), and $l^1_{\eps,+}(x,\xi)$ is the length of the geodesic starting from $(x,\xi)$ in $M_\eps$. Note that we would formally need to remove the set of trapped geodesics when applying Santalò's formula. However, as mentionned in the Remark \ref{rem:rem}, they have zero measure and do not influence the computation, so we forget them in order not to complicate the notations. By parametrizing each fiber $\partial_-S^*_xM^1_\eps$ with an angle $\omega \in [0,\pi]$, we can still disintegrate the measure $d\mu_{1,\nu}=\sin(\omega)d\omega dx$, where $dx$ is the measure induced by the metric $g_1$ on $\partial M_\eps$ and $d\omega$ is the measure in the fiber $\partial_-S^*M^1_\eps$, so that:
\[ \begin{split}
& \int_{S^*M^1_\eps}  \int_0^\pi J(f(x,\xi,\theta)) \sin(\theta) d\theta d\mu_1(x,\xi)  \\ &=  \int_{\partial M_\eps} \int_0^\pi \int_{0}^{l_{\eps,+}^1(x,\xi)} \int_0^\pi J(f(\varphi_\tau^1(x,\xi),\theta)) \sin(\theta) d\theta d\tau \sin(\omega) d\omega dx \\
& = \int_{\partial M_\eps} \int_{\omega_\eps}^{\pi-\omega_\eps} \int_{0}^{l_{\eps,+}^1(x,\xi)} \int_0^\pi J(f(\varphi_\tau^1(x,\xi),\theta)) \sin(\theta) d\theta d\tau \sin(\omega) d\omega dx  + ||J||_{L^\infty}\mathcal{O}(\eps^N), \end{split} \]
Recall that we applied Lemma \ref{lem:est1} with $\omega_\eps = \mathcal{O}(\eps^{N+1})$. The loss of $1$ in the exponent is due to the fact that we have to swallow uniformly the lengths $l^{1}_{\eps,+}(x,\xi) = \mathcal{O}(-\ln \eps)$ in the integral.

Let us fix $(x,\xi(\omega)) \in \partial_-S^*M^1_\eps \setminus \Gamma_-$ and consider one of its lift on the universal cover $(\wt{x},\wt{\xi}(\omega)) \in \partial_-S^*\wt{M}^1_\eps \setminus \wt{\Gamma}^1_-$. It generates a geodesic with endpoints $(p,q) \in \partial \wt{M} \times \partial \wt{M}$. We can rewrite the integral
\[ \int_{0}^{l_{\eps,+}^1(x,\xi)} \int_0^\pi J(f(\varphi_\tau^1(x,\xi),\theta)) \sin(\theta) d\theta d\tau  =  \int_{0}^{\wt{l}_{\eps,+}^1(\wt{x},\wt{\xi})} \int_0^\pi J(\wt{f}(\wt{\varphi}_\tau^1(\wt{x},\wt{\xi}),\theta)) \sin(\theta) d\theta d\tau, \]
We will now use the diffeomorphisms $\phi_i : V_i \rightarrow \mathcal{F}(p,q)$ (for $i = 1,2$) introduced in Section \ref{sect:liouv} (see equation (\ref{eq:v})). The $\wt{g}_1$-geodesic joining $p$ to $q$ is denoted by $\alpha_1(p,q)$: we choose a parametrization $\gamma : \R \rightarrow \alpha_1(p,q)$ by arc-length using the middle point (see Section \ref{sect:liouv}). Remark that the composition $\phi_2^{-1} \circ \phi_1 : V_1\rightarrow V_2$ has the form $(\tau,\theta) \mapsto (~\cdot~, \wt{f}(\gamma(\tau),\dot{\gamma}(\tau),\theta))$ (the first coordinate is of no interest to us). Moreover,
\[ \left(\phi_2^{-1} \circ \phi_1\right)^*\sin(\theta)d\theta d\tau = \phi_1^* \eta_2 = \phi_1^* \eta_1 = \sin(\theta)d\theta d\tau, \]
since the two Liouville currents agree according to Lemma \ref{lem:agree}. We have:
\[ \begin{split} \int_{0}^{\wt{l}_{\eps,+}^1(\wt{x},\wt{\xi})} \int_0^\pi  J(\wt{f}(\wt{\varphi}_\tau^1(\wt{x},\wt{\xi}),\theta)) \sin(\theta) d\theta d\tau  &  = \phi_1^* \eta_1 (J \circ \phi_2^{-1} \circ \phi_1 \cdot \mathbf{1}_{[T,T+\wt{l}_{\eps,+}^1(\wt{x},\wt{\xi})]\times [0,\pi]}) \\
& = \eta_1(J \circ \phi_2^{-1} \cdot \mathbf{1}_{\mathcal{F}_1(\wt{x},\wt{y})}) \\
& = \eta_2(J \circ \phi_2^{-1} \cdot \mathbf{1}_{\mathrm{h}(\mathcal{F}_1(\wt{x},\wt{y}))}) \\
& = \eta_2(J \circ \phi_2^{-1} \cdot \mathbf{1}_{\mathcal{F}_2(\wt{x}',\wt{y}')}) + ||J||_{L^\infty} \mathcal{O}(\eps^\infty) \\
& = \int_0^{\wt{l}_{\eps,+}^2(\wt{x}',\wt{\xi}')} \int_0^\pi J(\theta) \sin(\theta) d\theta d\tau + ||J||_{L^\infty} \mathcal{O}(\eps^\infty)  \\
& = \wt{l}_{\eps,+}^2(\wt{x}',\wt{\xi}') \int_0^\pi J(\theta) \sin(\theta) d\theta + ||J||_{L^\infty} \mathcal{O}(\eps^\infty), \end{split} \]
where the fourth equality follows from Lemma \ref{lem:est5}. The constant $T$ on the second line is unknown and appears in the choice of parametrization of the geodesic segment $\alpha_1(\wt{x},\wt{y})$ but does not influence the computation. The point $(\wt{x}',\wt{\xi}')=\wt{\psi}_\eps(\wt{x},\wt{\xi})$ is the image of $(\wt{x},\wt{\xi})$ by the diffeomorphism $\wt{\psi}_\eps$ defined in Section \ref{ssect:psi}. We recall that this diffeomorphism is invariant by the fundamental group and descends on the base as $\psi_\eps$.

Inserting this into the previous integrals, we obtain:
\[ \begin{split} &\int_{S^*M^1_\eps}  \int_0^\pi  J(f(x,\xi,\theta)) \sin(\theta) d\theta d\mu_1(x,\xi) \\
& =\int_0^\pi J(\theta) \sin(\theta) d\theta  \int_{\partial M_\eps} \int_{\omega_\eps}^{\pi-\omega_\eps} l_{\eps,+}^2(\psi_\eps(x,\xi(\omega))) \sin(\omega) d\omega dx + ||J||_{L^\infty} \mathcal{O}(\eps^N) \end{split} \]

According to Lemma \ref{lem:est2}, we know that $\psi_\eps = \text{Id} + \mathcal{O}(\eps^\infty)$ in the $\mathcal{C}^1$ topology. In particular, the Jacobian of $\psi_\eps$ is $1+\mathcal{O}(\eps^\infty)$ and by a change of variable: 
\[  \begin{split} \int_{\partial M_\eps} \int_{\omega_\eps}^{\pi-\omega_\eps} l_{\eps,+}^2(\psi_\eps(x,\xi(\omega))) \sin(\omega) d\omega dx & =  \int_{\partial M_\eps} \int_{0}^{\pi} l_\eps^2(x',\xi') \sin(\omega') d\omega' dx' + \mathcal{O}(\eps^N)\\
& = \text{vol}_{g_2}(S^*M^2_\eps) + \mathcal{O}(\eps^N) \\
& =  \text{vol}_{g_1}(S^*M^1_\eps) + \mathcal{O}(\eps^N),\end{split} \]
where the two volumes agree to order $\mathcal{O}(\eps^N)$ according to the same computation with $J \equiv 1$. Inserting this into the previous integrals, we obtain the sought result.

\end{proof}

Remark that we can actually consider in Lemma \ref{lem:est4} a family of functions $J_\eps$, instead of a single function. We can assume that $||J_\eps||_{L^\infty}=\mathcal{O}(1/\eps^{\alpha})$, for some $\alpha > 0$ which we may take as large as we want. Then, we can always apply the lemma with $N' := N + \floor*{\alpha}+1$, so that in the end, the sup norm $||J_\eps||_{L^\infty}$ is swallowed in the term $\mathcal{O}(\eps^N)$. We actually obtain for free a better version:

\begin{lemm}
\label{lem:convbien}
Let $N \in \N^*$ be an integer and $\alpha > 0$. Let $J_\eps : [0, \pi] \rightarrow \R$ be a family of convex continuous function such that $||J_\eps||_{L^\infty} = \mathcal{O}(\eps^{-\alpha})$. Then:
\be \label{eq:conv} \int_0^\pi J_\eps(\Theta_\eps(\theta)) \sin(\theta) d\theta \leq \int_0^\pi J_\eps(\theta) \sin(\theta) d\theta + \mathcal{O}(\eps^N), \ee
where the remainder only depends on $\eps$.
\end{lemm}

\section{Estimating the average angle of deviation}

\label{sect:trap}

As mentioned previously, we are unable to prove a priori that the $\Theta_\eps$ are uniformly Lipschitz. Nevertheless, we can show that they decompose as a sum $\Theta_\eps^{(a)}+\Theta_\eps^{(b)}$ where the $\Theta_\eps^{(a)}$ are Lipschitz (and their Lipschitz constant is controlled) and the $\Theta_\eps^{(b)}$ have a "small" $\mathcal{C}^0$ norm. This will be sufficient to apply our version of Otal's estimate (see Proposition \ref{prop:ot2}).

Note that we will sometimes drop the notation $C$ for the different constants which may appear at each line of our estimates and rather use the symbol $\lesssim$. By $||A|| \lesssim ||B||$, we mean that there exists a constant $C > 0$, which is independent of the elements $A$ and $B$ considered and such that, $||A|| \leq C ||B||$.

\subsection{Derivative of the angle of deviation}

The purpose of this paragraph is to estimate the derivative (with respect to $\theta$) of the angle of deviation $f$. We recall that
\[ W_1= \left\{(x,\xi,\theta) \in S^*M_1 \times (0,\pi), (x,\xi), (x,R_\theta \xi) \notin (\Gamma^1_- \cup \Gamma^1_+)\right\} \]


\begin{lemm}
\label{lem:df}
There exist constants $C, k > 0$ (independent of $\eps$) such that for all $(x,\xi,\theta) \in S^*M^{1}_\eps \cap W_1$:
\[ \left|\dfrac{\partial f}{\partial \theta}(x,\xi,\theta)\right| \leq C \exp{\left(k\left(l_{\eps, +}^{1}(x,R_\theta \xi) + |l_{\eps, -}^{1}(x,R_\theta \xi)| \right)\right)}   \]
\end{lemm}

\begin{proof}
We can write the derivative of $f$ as:
\be \label{eq:df} \dfrac{\partial f}{\partial \theta} = \dfrac{\partial f}{\partial y'} \left(\dfrac{\partial y'}{\partial \theta}\right) + \dfrac{\partial f}{\partial y} \left(\dfrac{\partial y}{\partial \theta}\right), \ee
where $y$ and $y'$ are defined in Section \ref{ssect:kappa} and study the different terms separately.

The idea is to study the behaviour (and more precisely the growth) of Jacobi vector fields in a vicinity of the boundary. Given a geodesic which enters the set $\left\{ \rho \geq \eps \right\}$, we will use the bounds (\ref{eq:growth}) to estimate the Jacobi vector fields on the segment contained in $\left\{\rho \geq \eps\right\}$. Then, by convexity, the geodesic exits $\left\{ \rho \geq \eps\right\}$ with a coordinate $\overline{\xi}_0 \leq 0$. On the set $\mathcal{C} = \left\{\rho < \delta \right\} \cap \left\{\overline{\xi}_0\leq 0\right\}$ (for some $\delta > 0$ small enough), we can study the behaviour of the geodesics more explicitly. Namely, given any point $(x,\xi) \in S^*M$ in $\mathcal{C}$, we know that it converges uniformly exponentially fast to the boundary in the sense that there exists $C > 0$ (uniform in $(x,\xi)$) such that if $\rho(t):=\rho(\varphi_t(x,\xi))$, then one has $\rho(0)e^{-t} \leq \rho(t) \leq C \rho(0)e^{-t}$ for $t\geq 0$ (see \cite[Lemma 2.3]{ggsu}). From the expression of the metric (\ref{eq:g}) in local coordinates, one can check that the curvature is given by $\kappa = -1 + \rho \cdot \mathcal{O}(1)$. As a consequence, if $\kappa(t) = \kappa(\pi_0(\varphi_t(x,\xi)))$ and $\delta > 0$ is chosen small enough at the beginning, one has that $-1-\frac{1}{10} e^{-t} \leq \kappa(t) \leq -1+\frac{1}{10} e^{-t}$, for any such $(x,\xi)$. If $t \mapsto \gamma(t)$ denotes the geodesic generated by this point and $J$ is a normal Jacobi vector field along $\gamma$, we write $J(t) = j(t) R_{\pi/2} \dot{\gamma}(t)$, where $j$ satisfies the Jacobi equation $\ddot{j}(t) + \kappa(t)j(t) = 0$. Assume $j(0) = 0,\dot{j}(0)=1$, then $j(t) > 0$ (there are no conjugate points) and thus $\ddot{j}(t) \leq (1+\frac{1}{10} e^{-t})j(t)$. By a comparison argument, $j(t) \leq z(t)$ where $z$ is the solution to $\ddot{z}(t)-(1+\frac{1}{10}e^{-t})z(t)=0$ with $z(0)=j(0), \dot{z}(0) = \dot{j}(0)$.

But making the change of variable $u = 2\sqrt{10}e^{-t/2}$, $\wt{z}(u) = z(t)$, one can prove that $\wt{z}$ solves the modified Bessel equation of parameter $2$ that is
\[  u^2 \dfrac{d^2 \wt{z}}{du^2} + u \dfrac{d\wt{z}}{du} - (u^2+2^2) \wt{z} = 0 \]
and thus $\wt{z}(u) = A \cdot I_2(u) + B \cdot K_2(u)$ for some parameters $A, B \in \R$ depending on $\wt{z}(0)$, $\dot{\wt{z}}(0)$, $I_2$ and $K_2$ being the modified Bessel functions of first and second kind. Thus: $z(t) = A \cdot I_2(2\sqrt{10}e^{-t/2}) + B \cdot K_2(2\sqrt{10}e^{-t/2})$ where $I_2(2\sqrt{10}e^{-t/2}) \sim_{t \rightarrow +\infty} C e^{-t}, K_2(2\sqrt{10}e^{-t/2}) \sim_{t \rightarrow + \infty} C e^{t}$ (see \cite[9.6.7-9.6.9]{as}) For instance, if $j(0)=0, \dot{j}(0) = 1$, which corresponds to a vertical variation of geodesics, then we obtain $|d\pi \circ d\varphi_t (V)| = |J(t)| \leq Ce^{t}$ for some constant $C > 0$ independent of the point. Using this technique of comparison and decomposing any vector by its vertical and horizontal components, one obtains that $||d\varphi_t(x,\xi)|| \leq Ce^{t}$ for $(x,\xi) \in \mathcal{C}$, where the constant $C > 0$ is uniform in $(x,\xi)$.

We fix $(x_0,\xi_0, \theta_0)$ and look at the variation $\theta \mapsto (x_0,R_{\theta_0+\theta} \xi_0)$. For each $\theta$, we thus have a $g_1$-geodesic $t \mapsto \gamma_\theta(t)$ generated by this point and it hits the boundary in the future at $y'(\theta)$. We set $\gamma := \gamma_0$. We denote by $J(t) := \partial_\theta \gamma_\theta(t)$ the Jacobi vector field along $\gamma$. Writing in short $l^{1}_{+,\eps} = l^{1}_{+,\eps}(x_0,R_{\theta_0}\xi_0), V=V(x_0,R_{\theta_0}\xi_0)$, we have for $t=s + l_\eps, s\geq0$:
\[  |J(t)|_{g_1} = \left|d\pi \circ d\varphi_{s+l^{1}_{+,\eps}}(V)\right|  \leq C e^{s} |d \pi \circ d\varphi_{l_\eps}(V)| \leq C e^{s} e^{k l^{1}_{+,\eps}}  \]
The first inequality follows from our previous remarks whereas the second one is a consequence of (\ref{eq:growth}). Now, we know that $\rho(l^1_{+,\eps})e^{-s} = \eps e^{-s} \leq \rho(t) \leq C \eps e^{-s} = C \rho(l^1_{+,\eps})e^{-s}$. As a consequence, for $t$ large enough, we have: $|J(t)|_{\overline{g}_1} = \rho(t) |J(t)|_{g_1} \leq C \cdot \eps e^{k l^{1}_{+,\eps}}$. By making $t \rightarrow +\infty$, we obtain that $\left| \dfrac{\partial y'}{\partial \theta} \right|_{h} \leq C\cdot \eps e^{k l^{1}_{+,\eps}}$.

Conversely, we consider a family of points $y'(u)$ in a vicinity of $y'_0$ on the boundary (such that $\left|\dfrac{\partial y'}{\partial u}\right|_{h} = 1$) and we look at the $g_2$-geodesics joining $y$ to $y'(u)$. They intersect the $g_2$-geodesic joining $z$ to $z'$ (the endpoints of the geodesic genreated by $(x,\xi)$) at some point $\mathrm{x}(u)$, and we obtain $(\mathrm{x}(u),\Xi(u))$ and an angle $f(u)$. From another perspective, we have a family of points $(\mathrm{x}(u),R_{f(u)}\Xi(u))$ which generate geodesics joining $y'(u)$ (in the future) to $y$ (in the past). Like before, we denote by $\gamma$ the geodesic obtained for $u=0$ and by $J$ the Jacobi vector field along $\gamma$. Since the point $y$ joined in the past by the geodesic is fixed (it does not depend on $u$), $J$ (more precisely, its lift in $TS^*M$) lies in the unstable bundle. We write
\[ \partial_u (\mathrm{x}(u),R_{f(u)}\Xi(u)) = d\pi^{-1}(J(0))+\mathcal{K}^{-1}(\nabla_t J(0))=\lambda \cdot \xi_u,\]
where $\xi_u$ is one of the two unit vectors (with respect to the $g_2$-Sasaki metric) generating $E_u$.
Note that the vertical component of this vector is precisely $\dfrac{\partial f}{\partial u} V$ and thus $|\lambda| \geq \left|\dfrac{\partial f}{\partial u}\right|$. We write $l^2_{+,\eps} = l^2_{+,\eps}(\mathrm{x},R_f \Xi)$. For $t=s+l^2_{+,\eps}, s\geq0$: 
\[ \begin{split} |J(t)|_{g_2} = |d\pi \circ d\varphi_t(\lambda \xi_u)| & = |\lambda| \cdot |d\pi \circ d\varphi_s\left(d\varphi_{l^2_{+,\eps}}(\xi_u)\right)| \\ 
& \geq |\lambda| \cdot e^{s} |d\varphi_{l^2_{+,\eps}}(\xi_u)| \\
& \geq C |\lambda| e^{s} e^{k l^2_{+,\eps}}\geq C \left|\dfrac{\partial f}{\partial u}\right| e^s e^{k l^2_{+,\eps}}\end{split}  \]
The term in $e^{k l^2_{+,\eps}}$ follows from (\ref{eq:growth}) whereas the term $e^s$ is a consequence on the bounds of the curvature. More precisely, for fixed bounds, that is $-k_0^2\leq\kappa\leq-k_1^2$, such a lower bound is obtained in \cite[Theorem 3.2.17]{wk2}, and the same proof applies here, except that we have bounds $-1-\frac{1}{10} e^{-t} \leq \kappa(t) \leq -1+\frac{1}{10} e^{-t}$. But the argument of Klingenberg is based on Gronwall lemma and $t \mapsto e^{-t}$ is integrable, so we get the same result in the end. Multiplying by $\rho(t)$ and taking the limit as $t \rightarrow +\infty$, we eventually obtain that $\left| \dfrac{\partial y'}{\partial u} \right|_{h} = 1 \geq C \eps e^{k l^2_{+,\eps}} \left|\dfrac{\partial f}{\partial u}\right|$.

Putting the previous bounds together, and using (\ref{eq:df}), we obtain the sought result.

\end{proof}

\subsection{Derivative of the exit time}

We set $T_\eps = -\wt{N} \ln \eps$ for some integer $\wt{N}$, like in the proof of Lemma \ref{lem:escape}.

\begin{lemm}
\label{lem:dl}
There exist constants $C, k > 0$ (independent of $\eps$) such that for all $(x,\xi,\theta) \in S^*M^1_\eps \cap W_1$ such that 
\[ T_\eps \leq l^{1}_{\eps,+}(x,R_\theta \xi) + |l^{1}_{\eps,-}(x,R_\theta \xi)|, \]
one has:
\[ \partial_\theta \left(l^{1}_{\eps,+}(x,R_\theta \xi) + |l^{1}_{\eps,-}(x,R_\theta \xi)|\right) \leq C \exp{\left(k\left(l^{1}_{\eps,+}(x,R_\theta \xi) + |l^{1}_{\eps,-}(x,R_\theta \xi)|\right)\right)} \]
\end{lemm}

\begin{proof}
Let us deal with the case of the exit time in the future, the other case being similar. The exit time is defined by the implicit equation:
\[ \rho\left(\varphi^{1}_{l^{1}_{\eps,+}(x,R_\theta \xi)}(x,R_\theta \xi)\right) = \eps \]
Differentiating with respect to $\theta$, we obtain:
\[ \partial_\theta \left(l^{1}_{\eps,+}(x,R_\theta \xi)\right) d\rho\left(X_{1}(\varphi^{1}_{l^{1}_{\eps,+}(x,R_\theta \xi)}(x,R_\theta \xi))\right) + d\rho \left(d\left(\varphi^{1}_{l^{1}_{\eps,+}(x,R_\theta \xi)}\right)_{(x,R_\theta \xi)}V(x,R_\theta \xi) \right) = 0, \]
where $V(x,\xi) \in \mathcal{V}$ is the vertical vector in $(x,\xi)$ (it is unitary with respect to the Sasaki metric $G_1$). But:
\[ \left|d\rho\left(X_{1}(\varphi^1_{l^{1}_{\eps,+}(x,R_\theta \xi)}(x,R_\theta \xi))\right)\right| = \eps |d\rho(\overline{X}_1)|, \]
and $d\rho(\overline{X}_1)$ is the sine of the angle with which the geodesic exits the region $\left\{\rho \geq \eps\right\}$. If this angle is less than $\frac{1}{10}$ (any small constant works as long as the geodesics concerned stay in a region where the metric still has the usual expression (\ref{eq:g})), then the geodesic will spend at most a bounded (independently of $\eps$) amount of time in the region $\left\{\rho \geq \eps\right\}$, thus contradicting the condition:
\[ T_\eps = -\wt{N} \ln(\eps) \leq l^{1}_{\eps,+}(x,R_\theta \xi) + |l^{1}_{\eps,-}(x,R_\theta \xi)| \] This can be proved using the Hamilton's equations, similarly to the proof of Lemma \ref{lem:est1} for instance. Thus $|d\rho(\overline{X}_1)| \geq \frac{1}{10}$.

As to the second term, using the fact that $d\rho/\rho$ is unitary (with respect to the dual metric of $g_1$ on the cotangent space), we obtain that:
\[ \begin{split} \left| \rho \dfrac{d\rho}{\rho} \left(d\left(\varphi^{1}_{l^{1}_{\eps,+}(x,R_\theta \xi)}\right)_{(x,R_\theta \xi)}V(x,R_\theta \xi) )\right) \right|
& \leq \eps \left| d\left(\varphi^{1}_{l^{1}_{\eps,+}(x,R_\theta \xi)}\right)_{(x,R_\theta \xi)}V(x,R_\theta \xi) \right|_{G_1} \\
& \leq \eps e^{k l^{1}_{\eps,+}(x,R_\theta \xi)}, \end{split}\]
for some constant $k$, following (\ref{eq:growth}). This provides the sought result.
\end{proof}

\subsection{An inequality on the average angle of deviation}

We know that $f$ is almost everywhere continuous and bounded, so $\Theta_\eps$ is continuous by Lebesgue theorem. We now prove that the homeomorphism $\Theta_\eps$ satisfies the following estimate:

\begin{lemm}
\label{lem:esttheta}
For any $\delta \in (Q,0)$ (defined in Lemma \ref{lem:escape}), for all $\beta > 0$ small enough, there exists $\beta' > 0$ (depending on $\beta$ and converging towards $0$ as $\beta \rightarrow 0$) such that:
\[ \forall \theta_1,\theta_2 \in [0,\pi], ~~ |\Theta_\eps(\theta_1)-\Theta_\eps(\theta_2)| \lesssim \eps^{-\beta'}|\theta_1-\theta_2|^{\beta} + \eps^{\delta} \]
\end{lemm}

\begin{proof}
First, remark that it is sufficient to prove the lemma for $\theta_1,\theta_2 \in [0,\pi/2]$, since the result will follow from the $\pi$-symmetry of the homeomorphism $\Theta_\eps$. We fix $\eps > 0$. We introduce the smooth cutoff function $\chi_T$ (for some $T > 0$ which will be chosen to depend on $\eps$ later) such that $\chi_T(s) \equiv 1$ on $[0,T]$ and $\chi_T(s) \equiv 0$ on $[2T,+\infty)$. Note that we can always construct such a $\chi_T$ so that $||\partial_s \chi_T||_{L^{\infty}} \leq 1$ (as long as $T > 1$, which we can assume since it will be chosen growing to infinity as $\eps \rightarrow 0$). We write $\Theta_\eps = \Theta_\eps^{(a), T} + \Theta_\eps^{(b),T}$, where:
\[ \begin{split} \Theta_\eps^{(a),T}(\theta) & := \dfrac{1}{\text{vol}_{g_1}(S^*M^{1}_\eps)} \int_{S^*M^{1}_\eps}  \chi_T\left(l^{1}_{\eps,+}(x,R_\theta\xi) + |l^{1}_{\eps,-}(x,R_\theta\xi)|\right)  f(x,\xi,\theta) d\mu_1(x,\xi) \\
& = \dfrac{1}{\text{vol}_{g_1}(S^*M^{1}_\eps)} \int_{S^*M^{1}_\eps} \psi_T(x,\xi,\theta)  \end{split} \]
where $\psi_T$ is defined to be the integrand and
\[ \Theta_\eps^{(b),T}(\theta) := \Theta_\eps - \Theta_\eps^{(a),T} \]
Morally, the cutoff function mean that we integrate over the compact region
\[ \left\{l^{1}_{\eps,+}(x,R_\theta\xi) + |l^{1}_{\eps,-}(x,R_\theta\xi)| \leq T\right\}\]
By Lebesgue theorem, $\Theta_\eps^{(a),T}$ is $\mathcal{C}^1$ on $[0,\pi/2]$. For $\beta > 0$, $\theta_1, \theta_2 \in [0,\pi/2]$, one has:
\[ |\Theta^{(a),T}_\eps(\theta_1) - \Theta^{(a),T}_\eps(\theta_2)| \lesssim \sup_{\theta \in [0,\pi/2]} \left|\partial_\theta \Theta^{(a),T}_\eps(\theta)\right|^{\beta} |\theta_1-\theta_2|^{\beta} \]
Let us estimate the former derivative. We have:
\[ \partial_\theta \Theta^{(a),T}_\eps(\theta) = \dfrac{1}{\text{vol}_{g_1}(S^*M^{1}_\eps)} \int_{S^*M^{1}_\eps} \partial_\theta \psi_T(x,\xi,\theta) d\mu_1(x,\xi), \]
and the derivative under the integral is composed of a sum of two terms which we now estimate separately.

\begin{enumerate}
\item By Lemma \ref{lem:df}, the first term is bounded by:
\[  \left|\chi_T\left(l^{1}_{\eps,+}(x,R_\theta\xi) + |l^{1}_{\eps,-}(x,R_\theta\xi)|\right) \partial_\theta f(x,\xi,\theta)\right| \lesssim \exp{\left(k(l^{1}_{\eps,+}(x,R_\theta\xi) + |l^{1}_{\eps,-}(x,R_\theta\xi)|)\right)} \lesssim  e^{2kT} \]

\item And the second term is bounded by Lemma \ref{lem:dl}:
\[  \left| \partial_\theta \left(l^{1}_{\eps,+}(x,R_\theta\xi) + |l^{1}_{\eps,-}(x,R_\theta\xi)\right) \partial_s \chi_T\left(l^{1}_{\eps,+}(x,R_\theta\xi) + |l^{1}_{\eps,-}(x,R_\theta\xi)|\right)  f(x,\xi,\theta)\right| \lesssim e^{2kT}\]

\end{enumerate}

Note that the constant $k >0$ may be different from one line to another. Gathering everything, we obtain that for all $\theta \in [0,\pi/2]$, $|\partial_\theta \Theta^{(a),T}_\eps(\theta)| \lesssim e^{2kT}$ and thus:
\[ |\Theta^{(a),T}_\eps(\theta_1) - \Theta^{(a),T}_\eps(\theta_2)| \lesssim e^{2k\beta T} |\theta_1-\theta_2|^{\beta} \]

As to $\Theta^{(b),T}_\eps$, we can write:
\[ \begin{split} \Theta^{(b),T}_\eps(\theta) & \leq \dfrac{1}{\text{vol}_{g_1}(S^*M^{1}_\eps)}\left( \int_{S^*M^{1}_\eps \cap \left\{l^{1,+}_\eps(x,R_\theta \xi) > T\right\}} f d\mu_1 + \int_{S^*M^{1}_\eps \cap \left\{|l^{1,-}_\eps(x,R_\theta \xi)| > T\right\}} f d\mu_1 \right) \end{split} \]
If $T \geq -\wt{N}\ln(\eps)$ ($\wt{N}$ is a large integer defined in Lemma \ref{lem:escape}, independent of $\eps$), then the two integrals can be estimated by Lemma \ref{lem:escape} (note that we here divide by the volume which is bounded by $\mathcal{O}(\eps)$). We obtain:
\[ |\Theta^{(b),T}_\eps(\theta)| \lesssim e^{-\delta T} \eps^{-4\delta} \]
We choose $T := T_\eps = -\wt{N}\ln(\eps)$ and set $\Theta_\eps^{(a)}:=\Theta_\eps^{(a),T_\eps},\Theta_\eps^{(b)}:=\Theta_\eps^{(b),T_\eps}$. Since $\wt{N}$ is taken large enough (greater than $5$ at least to swallow the $\eps^{-4\delta}$), we obtain $||\Theta^{(b)}_\eps||_{L^\infty} \lesssim \eps^{\delta}$. And :
\[ |\Theta^{(a),T}_\eps(\theta_1) - \Theta^{(a),T}_\eps(\theta_2)| \leq \eps^{-2\beta k \wt{N}} |\theta_1-\theta_2|^{\beta}, \]
which provides the sought result by going back to $\Theta_\eps$.
\end{proof}

\subsection{Otal's lemma revisited}

In the spirit of Otal's lemma (see \cite[Lemma 8]{jo}), we prove:

\begin{prop}
\label{prop:ot2}
Assume $\Theta_\eps : [0,\pi] \rightarrow [0,\pi]$ is a family of increasing homeomorphisms for $\eps \in (0,\delta)$ such that:
\begin{enumerate}
\item $\Theta_\eps(0)=0, \Theta_\eps(\pi)=\pi$,
\item For all $\theta \in [0,\pi], \Theta_\eps(\pi-\theta) = \pi-\Theta_\eps(\theta)$,
\item For all $\theta_1,\theta_2\in[0,\pi]$ such that $\theta_1+\theta_2\in[0,\pi]$, 
\[\Theta_\eps(\theta_1)+\Theta_\eps(\theta_2) \leq \Theta_\eps(\theta_1+\theta_2)\]
\item There exists $\alpha > 2\beta'/\beta-1$ ($\beta$ and $\beta'$ are defined afterwards) such that for all family of continuous convex functions $J_\eps : [0,\pi] \rightarrow \R$ such that $||J_\eps||_{L^\infty} = \mathcal{O}(1/\eps^\alpha)$,
\[ \int_0^\pi J_\eps(\Theta_\eps(\theta))\sin(\theta)d\theta \leq \int_0^\pi J_\eps(\theta)\sin(\theta)d\theta + \mathcal{O}(\eps) \]
\item There exists constants $C, \beta, \beta' > 0$ and $\delta > 0$ (independent of $\eps$), such that for all $\theta_1,\theta_2 \in [0,\pi]$:
\[ |\Theta_\eps(\theta_1)-\Theta_\eps(\theta_2)| \leq C\left(\eps^{\delta} + \eps^{-\beta'}|\theta_1-\theta_2|^{\beta}\right) \]
\end{enumerate}
Then $\Theta_\eps = \text{Id} + \mathcal{O}(\eps^{\gamma})$, where we can take any $\gamma$ up to the critical exponent
\[ \hat{\gamma} := \dfrac{1+\alpha-2 \beta'/\beta}{1+2/\beta}, \]
as long as $\gamma < \delta$.
\end{prop}

\begin{proof}[Proof]
We argue by contradiction. Assume there there exists a sequence $\eps_n \rightarrow 0$ such that $||\Theta_n-\text{Id}||_{L^\infty} > n \eps_n^{\gamma}$ (where $\Theta_n := \Theta_{\eps_n}$). By $\pi$-symmetry, there exists an interval $[a_n, A_n]$ such that for all $\theta \in (a_n, A_n)$, $\Theta_n(\theta) < \theta - n\eps_n^{\gamma}$ and we can choose $\Theta_n(a_n) = a_n - n\eps_n^{\gamma}, \Theta_n(A_n)=A_n-n\eps_n^{\gamma}$.

We also construct the largest interval $[b_n,B_n] \supset [a_n,A_n]$ such that for all $\theta \in (b_n,B_n)$, $\Theta_n(\theta) < \theta - \eps_n^{\gamma}$ and $\Theta_n(b_n) = b_n-\eps_n^{\gamma}, \Theta_n(B_n)=B_n-\eps_n^{\gamma}$. Eventually, we define the largest interval $[c_n,C_n] \supset [b_n,B_n]$ such that for all $\theta \in (c_n,C_n)$, $\Theta_n(\theta) < \theta$ and $\Theta_n(c_n)=c_n, \Theta_n(C_n)=C_n$. The $\pi$-symmetry implies that $\Theta(\pi/2) = \pi/2$ and since $\Theta(0)=0, \Theta(\pi)=\pi$, we know that the points $c_n < b_n < a_n < A_n < B_n < C_n$ all lie either in $[0,\pi/2]$ or in $[\pi/2,\pi]$.

\begin{figure}[h!]
\begin{center}

\includegraphics[scale=0.8]{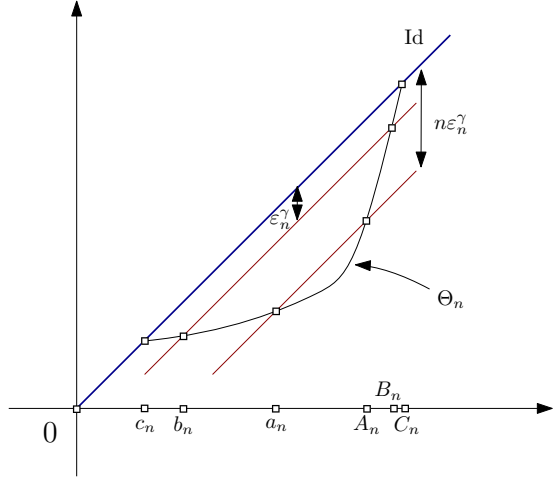} 
\caption{The points $c_n < b_n < a_n < A_n < B_n < C_n$}

\end{center}
\end{figure}

Remark that $\Theta_n-\text{Id}$ also satisfies the fifth item, namely:
\[ \begin{split} |(\Theta_n-\text{Id})(\theta_1) - (\Theta_n-\text{Id})(\theta_2)| & \lesssim |\Theta_n(\theta_1)-\Theta_n(\theta_2)| + |\theta_1-\theta_2| \\
& \lesssim \left(\eps_n^{\delta} + \dfrac{1}{\eps_n^{\beta'}}|\theta_1-\theta_2|^{\beta}\right) + (2\pi)^{1-\beta}|\theta_1-\theta_2|^{\beta} \\
& \lesssim \eps_n^{\delta} + \dfrac{1}{\eps_n^{\beta'}}|\theta_1-\theta_2|^{\beta}
\end{split} \]
This implies that:
\[ |(\Theta_n-\text{Id})(a_n)-(\Theta_n-\text{Id})(b_n)| = (n-1)\eps_n^{\gamma} \lesssim \eps_n^{\delta} + \dfrac{1}{\eps_n^{\beta'}}(a_n-b_n)^{\beta} \]
Thus:
\[ (a_n-b_n)^{\beta} \gtrsim (n-1)\eps_n^{\gamma+\beta'}-\eps_n^{\delta+\beta'} \gtrsim (n-1)\eps_n^{\gamma+\beta'}, \]
for $n$ large enough since $\delta > \gamma$. The same inequalities hold for the other points and we get, for $n$ large enough:
\[ a_n-b_n \gtrsim (n-1)^{1/\beta}\eps^{(\gamma+\beta')/\beta}_n, ~~~~ B_n - A_n \gtrsim (n-1)^{1/\beta}\eps^{(\gamma+\beta')/\beta}_n\]
\[ b_n-c_n \gtrsim \eps^{(\gamma+\beta')/\beta}_n, ~~~~C_n-B_n \gtrsim \eps^{(\gamma+\beta')/\beta}_n \]
Now, for $h \in (0,C_n-c_n)$, by superadditivity:
\[ c_n + h > \Theta_n(c_n+h) \geq \Theta_n(c_n) + \Theta_n(h) = c_n + \Theta_n(h), \]
that is $\Theta_n(h) < h$. In the same fashion, we have for $h \in (b_n-c_n,B_n-c_n), \Theta_n(h) < h-\eps_n^{\gamma}$.

Let us now consider the continuous convex functions $J_n(x):=\eps_n^{-\alpha} \sup(C_n-c_n-x,0)=\eps_n^{-\alpha}\wt{J}_n(x)$ on $[0,\pi]$. Using:
\[ \int_0^\pi \wt{J}_n(\Theta_n(\theta))\sin(\theta)d\theta \leq \int_0^\pi \wt{J}_n(\theta)\sin(\theta)d\theta + C\eps^{1+\alpha}_n, \]
where $C > 0$ is a constant independent of $n$, we obtain:
\[ \begin{split} 0 & \leq \int_0^{C_n-c_n} (\Theta_n(\theta)-\theta)\sin(\theta)d\theta + C \eps^{1+\alpha}_n \\
& = \int_0^{b_n-c_n} (\Theta(\theta)-\theta)\sin(\theta)d\theta + \int_{b_n-c_n}^{B_n-c_n} " + \int_{B_n-c_n}^{C_n-c_n} " + C\eps^{1+\alpha}_n  \\ 
& < C\eps^{1+\alpha}_n - \eps_n^{\gamma} \int_{b_n-c_n}^{B_n-c_n} \sin(\theta)d\theta, \end{split} \]
where we used the bounds stated above and the fact that both $b_n-c_n$ and $B_n-c_n$ are in $[0,\pi/2]$. But remark that:
\[\begin{split} \int_{b_n-c_n}^{B_n-c_n} \sin(\theta)d\theta & \geq  \left((B_n-c_n)-(b_n-c_n)\right)\sin(b_n-c_n) \\
& \geq C'(n-1)^{1/\beta}\eps_n^{2(\gamma+\beta')/\beta},\end{split} \]
for some constant $C' > 0$, by inserting the previous bounds and using the inequality $\sin(x)\geq 2x/\pi$ on $[0,\pi/2]$. Thus, we obtain:
\[ 0 < \eps_n^{1+\alpha}\left(C-C'(n-1)^{1/\beta}\eps_n^{(2/\beta +1)\gamma + 2\beta'/\beta-1-\alpha}\right), \]
and $(2/\beta +1)\gamma + 2\beta'/\beta-1-\alpha \leq 0$ by the definition of $\gamma$, so the right-hand side is negative as $n$ goes to infinity.
\end{proof}

\begin{rema} Let us mention that the result is still valid in the limit $\delta = +\infty, \beta=1,\beta'=0$ (the $\Theta_\eps$ are uniformly Lipschitz) and $\alpha=0$. It provides an exponent $\gamma=1/3$. Had we been able to prove a priori that the family $\Theta_\eps$ was uniformly Lipschitz, this would have been enough to conclude. \end{rema}

\section{End of the proof}

\label{sect:end}

We can now conclude the proof.

\begin{proof}[Proof] 

Combining Lemmas \ref{lem:convbien}, \ref{lem:esttheta} and Proposition \ref{prop:ot2}, we conclude that $\Theta_\eps = \text{Id}+\mathcal{O}(\eps^N)$, for some $N$ which we can choose large enough. Thus for $\theta_1,\theta_2 \in [0,\pi]$ such that $\theta_1+\theta_2\in[0,\pi]$:
\[ \begin{split} 0  \leq \dfrac{1}{\text{vol}(S^*\wt{M}^1_\eps)} & \int_{S^*\wt{M}^1_\eps} f(x,\xi,\theta_1+\theta_2) - f(x,\xi,\theta_1) - f(x,R_{\theta_1}\xi,\theta_2) ~ d\mu_1(x,\xi)  \\
&= \Theta_\eps(\theta_1+\theta_2)-\Theta_\eps(\theta_1)-\Theta_\eps(\theta_2) \\
& = \mathcal{O}(\eps^N) \end{split} \]
Since the integrand is positive and the inverse of the volume can be estimated by $\mathcal{O}(\eps)$, this implies by taking $\eps\rightarrow 0$ that 
\[ f(x,\xi,\theta_1+\theta_2) - f(x,\xi,\theta_1) - f(x,R_{\theta_1}\xi,\theta_2) = 0 \]
so the inequality is saturated in Gauss-Bonnet formula. As a consequence, three intersecting $g_1$-geodesics correspond to three intersecting $g_2$-geodesics with same endpoints.

We can now construct the isometry $\Phi$ between $(M,g_1)$ and $(M,g_2)$. We will use in this paragraph the notation $\wt{\cdot}$ to refer to objects considered on the universal cover $\wt{M}$. Given $p \in \wt{\overline{M}}$, we choose three $g_1$-geodesics $\alpha$, $\beta$ and $\gamma$ passing through $p$ with respective endpoints $(x,x')$, $(y,y')$ and $(z,z')$ in $\partial \wt{M} \times \partial \wt{M}$. By the previous section, we know that the $g_2$-geodesics with same endpoints meet in a single point which we define to be $\wt{\Phi}(p)$. Now, $\wt{\Phi}(p)$ is well-defined (it does not depend on the choice of the geodesics) and remark that for $(x,\xi) \notin \wt{\overline{\Gamma}}_- \cup \wt{\overline{\Gamma}}_+$ (such a covector always exist) and $\theta$ such that $(x,R_\theta \xi) \notin \wt{\overline{\Gamma}}_- \cup \wt{\overline{\Gamma}}_+$, we have $\wt{\Phi}(p) = \mathrm{x}(x,\xi,\theta)$, where $\mathrm{x}$ is defined in (\ref{eq:kappa}) (in other words, $\kappa$ maps fibers to fibers). Thus $\wt{\Phi}$ is $\mathcal{C}^\infty$ in the interior (see Section \ref{ssect:kappa}) and extends continuously down to the boundary as $\wt{\Phi}|_{\partial \wt{M}}=\text{Id}$.

Moreover, $\wt{\Phi}^*(\wt{g}_2)=\wt{g}_1$. Indeed, it is sufficient to prove that $\wt{\Phi}$ preserves the distance. Given $p,q \in \wt{M}$, we have $\wt{\mathcal{F}}_1(p,q) = \wt{\mathcal{F}}_2(\wt{\Phi}(p),\wt{\Phi}(q))$ and thus:
\[ d_{\wt{g}_1}(p,q)=\dfrac{1}{2}\eta_{\wt{g}_1}\left(\wt{\mathcal{F}}_1(p,q)\right) = \dfrac{1}{2}\eta_{\wt{g}_2}\left(\wt{\mathcal{F}}_2(\wt{\Phi}(p),\wt{\Phi}(q))\right)=d_{\wt{g}_2}(\wt{\Phi}(p),\wt{\Phi}(q)) \]

Now, observe that $\wt{\Phi}$ is invariant by the action of the fundamental group: it thus descends to a smooth diffeomorphism $\Phi : M \rightarrow M$ which extends continuously down to the boundary and satisfies $\Phi^*g_2=g_1$.

We now conclude the argument by proving that $\Phi$ is actually smooth on $\overline{M}$. In the compact setting, it is a classical fact that an isometry which is at least differentiable is actually smooth and our argument somehow follows the idea of proof of this statement. More precisely, we show that a smooth isometry on an asymptotically hyperbolic manifold actually extends as a smooth application on the compactification $\overline{M}$. The proof does not rely on the dimension two. Note that another proof could be given in this case using the fact that $\Phi$ is a conformal map.

Consider a fixed point $p \in M$ in a vicinity of the boundary. For any point $q \in \overline{M}$ in a vicinity of $p$, we denote by $\xi(q)$ the unique covector such that $w(q):=(p,\xi(q))$ generates the geodesic joining $p$ to $q$. The map $q \mapsto \xi(q)$ is smooth down to the boundary by \cite[Proposition 5.13]{ggsu}. Let us denote by $\tau_1(q)$ the time such that $q=\pi_0\left(\overline{\varphi}^1_{\tau_1(q)}(w(q))\right)$. It is smooth down to the boundary too. Since $\Phi$ conjugates the two geodesic flows, we can write:
\[ \Phi(q) = \pi_0\left(\overline{\varphi}^2_{\tau_2(q)}(z(q))\right), \]
where $z(q):=(\Phi(p),d\Phi_p(\xi(q)))$, for some time $\tau_2(q)$. All is left to prove, is thus that $\tau_2$ is smooth down to the boundary. If $t(q)$ denotes the $g_1$-geodesic distance between $p$ and $q$ (which is also that between $\Phi(p)$ and $\Phi(q)$ for $g_2$), one has:
\[ t(q) = \int_0^{\tau_1(q)} \dfrac{ds}{\rho(\overline{\varphi}^1_s(p,\xi(q)))} = -\ln\left(1-\dfrac{\tau_1(q)}{\tau^1_+(w(q))}\right) + G(\tau_1(q),w(q)), \]
for some smooth function $(\tau,z) \mapsto G(\tau,z)$ down to the boundary (this is a computation similar to the one carried out in Section \ref{sssect:def}, see also \cite[Lemma 2.7]{ggsu}). And:
\[ \tau_2(q) = \tau^{2}_+(z(q))-e^{-t(q)}\tau^{2}_+(z(q))H(e^{-t},z(q)), \]
for some smooth positive function $H$ on $[0,1) \times \overline{S^*M} \setminus \left(\partial_-S^*M\cup\Gamma_-\right)$ (this stems from the previous equality, or see also \cite[Lemma 2.7]{ggsu}). As a consequence:
\[ \tau_2(q) = \tau^{2}_+(z(q)) - \left(1-\dfrac{\tau_1(q)}{\tau^1_+(w(q))}\right)I(q), \]
for some smooth function $I$ down to the boundary, which can be expressed in terms of $H$ and $G$. This concludes the proof.

\end{proof}

\end{document}